\documentclass[11pt]{article}

\usepackage{lmodern}
\usepackage{csquotes}
\usepackage{amsmath,amssymb}
\usepackage{parskip}
\usepackage{graphicx}
\usepackage{booktabs, tabularx}
\usepackage{multirow}
\usepackage{hyperref}
\usepackage{float}
\usepackage{adjustbox}
\usepackage[labelfont=bf]{caption, subcaption}
\usepackage[dvipsnames, svgnames]{xcolor}
\usepackage[font=small]{caption}

\DeclareMathOperator{\dom}{dom}
\DeclareMathOperator{\R}{\mathbb{R}}
\DeclareMathOperator{\Z}{\mathbb{Z}}

\DeclareMathOperator{\D}{\mathcal{D}}
\DeclareMathOperator{\N}{\mathcal{N}}
\DeclareMathOperator{\A}{\mathcal{A}}
\newcommand{\X}{\mathcal{X}}
\DeclareMathOperator{\Y}{\mathbb{Y}}
\DeclareMathOperator{\Ycal}{\mathcal{Y}}
\newcommand{\F}{\mathcal{F}}

\renewcommand{\H}{\mathcal{H}}

\DeclareMathOperator{\bin}{\{0, 1\}}
\DeclareMathOperator{\st}{s.t.}

\DeclareMathOperator{\conv}{conv}

\DeclareMathOperator{\proj}{proj}

\DeclareMathOperator{\gph}{gph}
\DeclareMathOperator{\hyp}{hyp}
\DeclareMathOperator{\rec}{rec}
\newcommand{\dotcup}{\mathbin{\mathaccent\cdot\cup}}
\newcommand{\ie}{\emph{i.e.}}
\newcommand{\eg}{\emph{e.g.}}

\newcommand{\qedwhite}{\hfill \ensuremath{\Box}}
\newcommand{\xbar}{\widebar{x}}
\newcommand{\xhat}{\widehat{x}}

\newcommand{\yhat}{\widehat{y}}
\newcommand{\zbar}{\widebar{z}}
\newcommand{\zhat}{\widehat{z}}
\newcommand{\pihat}{\widehat{\pi}}
\newcommand{\lambdahat}{\widehat{\lambda}}
\newcommand{\gammahat}{\widehat{\gamma}}
\newcommand{\idx}{\texttt{i}}
\newcommand{\responseSet}{\mathcal{R}}

\definecolor{mygreen}{RGB}{0, 160, 0}
\definecolor{mypurple}{RGB}{204, 0, 204}
\definecolor{myorange}{RGB}{245, 118, 0}
\definecolor{myblue}{RGB}{0, 128, 255}
\definecolor{myyellow}{RGB}{196, 175, 24}

\usepackage{tikz, tikz-3dplot}
    \usetikzlibrary{arrows, shapes}
\usepackage{pgf}

\usepackage[letterpaper,margin=1in]{geometry}
\usepackage{natbib}
\usepackage{amsthm}
\usepackage{bm}
\usepackage{authblk}
\usepackage{microtype}
\hypersetup{hidelinks}
\DeclareMathOperator*{\argmin}{arg\,min}
\newcommand{\widebar}[1]{\overline{#1}}

\newtheorem{lemma}{Lemma}
\newtheorem{proposition}{Proposition}
\newtheorem{corollary}{Corollary}
\theoremstyle{definition}
\newtheorem{example}{Example}
\theoremstyle{remark}
\newtheorem{remark}{Remark}

\renewcommand{\proof}[1]{\par\noindent\textit{#1}\ }
\newcommand{\Halmos}{\ensuremath{\square}}

\title{Network flow-based cuts for binary bilevel programs}
\author[1]{Sebasti\'an V\'asquez\thanks{\texttt{savasque@andrew.cmu.edu}}}
\author[2]{Leonardo Lozano\thanks{\texttt{leolozano@uc.edu}}}
\author[1]{Willem-Jan van Hoeve\thanks{\texttt{vanhoeve@andrew.cmu.edu}}}
\affil[1]{Carnegie Mellon University}
\affil[2]{University of Cincinnati}
\date{}

\begin{document}
\maketitle

\begin{abstract}
Binary bilevel programs are challenging to solve due to the lack of compact single-level reformulations. Recent extended network flow-based reformulations have emerged as a promising approach, providing strong dual bounds. However, their increasing size can make them computationally prohibitive. We develop a cutting-plane framework that exploits these networks without requiring an extended-space formulation. We introduce two exact network flow-based reformulations and provide an efficient separation procedure for inequalities that certify the follower’s optimality. We further strengthen these inequalities through a tilting procedure that leverages the combinatorial structure of the networks. Computational experiments show that the proposed approach becomes increasingly advantageous as networks grow larger or more numerous. Moreover, incorporating the cuts into state-of-the-art bilevel solvers improves solution times and optimality gaps on BOBILib instances, closing eight more open instances than the baseline solvers. These results demonstrate the effectiveness of using network representations as separation oracles for binary bilevel optimization.
\end{abstract}

\noindent\textbf{Keywords:} Binary bilevel programming, cutting planes, network flows

\medskip
\noindent\textbf{Funding:} Dr. Lozano was supported by AFOSR grant FA9550-25-1-0301.

\medskip


\section{Introduction}\label{sec:intro}

Bilevel optimization represents hierarchical interactions between two decision-makers, commonly referred to as the leader and the follower, who optimize distinct objective functions subject to individual and coupling constraints \citep{bracken_mathematical_1973}. The leader has complete knowledge of the follower's problem, and only controls its own variables, while the follower selects an optimal response after observing the leader’s decisions. If the follower admits multiple optimal responses, under the optimistic setting the follower selects one that is most favorable to the leader, whereas under the pessimistic setting, it selects one that is least favorable. The overall problem consists in finding a leader decision that optimizes the leader’s objective while accounting for the follower’s optimal reaction.

For continuous bilevel programs, single-level reformulations can often be obtained by replacing the follower’s problem with its Karush-Kuhn-Tucker conditions or by exploiting strong duality. These reformulations typically yield nonlinear programs or mathematical programs with complementarity constraints~\citep{fortuny1981representation,bard1990branch,zare2019note}. Such techniques, however, do not directly apply to problems with integer follower decisions, which makes integer bilevel programs substantially more difficult to reformulate and solve. State-of-the-art exact methods for integer bilevel optimization generally extend mixed-integer programming techniques by deriving and separating logic-based and intersection cuts~\citep{fischetti2017flms,tahernejad2020mibs}. Other recent papers leverage network flow-based representations of the follower's problem to obtain single-level reformulations and relaxations, and have achieved competitive results~\citep{lozano_constrained_2022,LBC24,vasquez2025single}. The main source of computational overhead of the latter approaches is the impact of the size of the networks, as the reformulations are parametrized by the components of the network. Hence, if the network grows exponentially large, so will the reformulations.

In this work, we provide an alternative method to recover the objective value bound obtained with the decision diagram approach in \cite{LBC24} through valid inequalities. This allows to keep the single-level relaxation on the original variable space, improving its computational tractability. We first propose two new exact network flow-based single-level reformulations of the binary bilevel program. Then, we show how to separate the inequalities required to convexify the set of follower's optimal responses. This, in turn, allows to design a cutting plane approach that inherits the properties proposed in \cite{LBC24} to yield dual bounds by using size-controlled networks. Finally, inspired by the cut strengthening techniques in \cite{behle2007binary} and \cite{castro2022lifting}, we exploit the combinatorial structure of the networks to improve our cuts via tilting. We note that the proposed method is related to the one in \cite{lefebvre2025DW}, where the dual of a Dantzig-Wolfe representation of the follower's feasible solution set is used to build a bilinear single-level reformulation. In contrast, we work on an extended primal space induced by flow variables over networks. This allows to design exact linear reformulations, avoiding nonlinearities. Our method is also related to the one in \cite{chen2025recovering}, which focuses on recovering a Dantzig-Wolfe reformulation bound via linear inequalities. However, in the bilevel setting, the feasible solution set is not explicitly represented, hindering the direct construction of outer approximations. The network flow encoding provides a more compact approximation of the follower's optimal value, while exhibiting a combinatorial structure over which efficient graph algorithms can be executed to build and strengthen outer approximations of the feasible solution set.


\noindent \textbf{Contributions.}
Our first contribution is the design of two network flow-based single-level reformulations of binary bilevel programs. They replace the follower's optimality condition with a family of valid inequalities, for which we also provide options for their separation. Our second contribution is the strengthening of these inequalities via tilting. Inspired by the work of \cite{castro2022lifting}, we leverage the combinatorial structure of the networks to provide a relatively efficient way to tilt valid cuts. Lastly, our third contribution is a comprehensive computational study of the performance of the new cut families. In particular, we show that our cuts are especially effective for large networks or in a decomposition setting with multiple independent networks, and assess their portability to benchmark BOBILib instances by incorporating them into state-of-the-art solvers.


\noindent \textbf{Organization.}
The remainder of the paper is organized as follows. In Section \ref{sec:preliminaries}, we introduce relevant definitions, notation, and previous work on network flow-based reformulations for binary bilevel optimization. In Section \ref{sec:new-reformulations}, we present our new reformulations and cuts, and explain how the latter can be separated. In Section \ref{sec:cut-strengthening}, we describe how to use the networks to strengthen cuts. In Section \ref{sec:experiments}, we present numerical results assessing the effectiveness of our cuts. Finally, we conclude the work in Section \ref{sec:conclusions}.

\section{Preliminaries} \label{sec:preliminaries}
We consider the binary linear bilevel program
\begin{subequations} \label{model:BBP}
    \begin{alignat}{2}
        \min_{x, y} & & \hspace{.2cm} & c_L^\top x +c_F^\top y \\
        \st & & & Ax + By \leq a \\
        & & & x \in \bin^{n_L} \\
        & & & y \in \responseSet \left(x \right), \label{constr:BBP-ReactionSet}
    \end{alignat}
\end{subequations}
where
\[
    \responseSet \left(x\right) := \argmin_y \big\{ d^\top y :
    Dy \leq b - Cx, \, y \in \Ycal \subseteq \R^{n_F} \big\}
\]
%
is the follower's optimal response set,~$\Y(x) = \{y \in \Ycal : Dy \leq b - Cx\}$ is the parametric follower's feasible solution set, and~$\dom \Y = \{x \in \bin^{n_L}: \Y(x) \neq \emptyset\}$ is the set of binary leader decisions that render the follower's problem feasible. 
Note that program \eqref{model:BBP} can also model bounded integer leader variables $x$ through binary expansion. After applying a linear transformation to make the leader variables nonnegative, we introduce a new variable $x^j \in \bin^{\left\lfloor \log(u_j) \right\rfloor + 1}$ for each $j \in [n_L]$ such that $x_j \in \Z$, where $u_j > 0$ is the upper bound on $x_j$ after the transformation. This way, by replacing $x_j$ with $\sum_{i = 0}^{\left\lfloor \log(u_j) \right\rfloor} 2^i x^j_i$, we can retrieve a binary bilevel program of the form \eqref{model:BBP}.

We refer to function $\phi(x) = \min\{ d^\top y: y \in \Y(x)\}$ as the follower's value function, which computes the optimal value of the follower's problem for a given vector $x$. We assume that if $x \not\in \dom \Y$, then $\phi(x) = +\infty$.
The \textit{High-Point Relaxation} (HPR) corresponds to the program obtained from relaxing \eqref{constr:BBP-ReactionSet} from \eqref{model:BBP}, \ie,~$\min\{c_L^\top x + c_F^\top y: (x, y) \in \H \}$, where~$\H := \{(x, y) \in \bin^{n_L} \times \Ycal: Ax + By \leq a,\ Cx + Dy \leq b\}$.
Therefore, program \eqref{model:BBP} can be reformulated as
\begin{equation} \label{model:BBP-valueFunctionRef}
    \min\{ c_L^\top x + c_F^\top y : (x, y) \in \H,\ d^\top y \leq \phi(x) \},
\end{equation}
which is known as the \textit{value function reformulation}. We denote as~$\F := \{ (x, y) \in \H: d^\top y \leq \phi(x) \}$ the set of feasible solutions to \eqref{model:BBP-valueFunctionRef}, also known as \textit{bilevel-feasible set}.

\begin{example} \label{ex:example1}
    As a running example, consider the integer bilevel program in Figure~\ref{fig:example1-model} with follower's optimal response set $\responseSet(x)$ in Figure~\ref{fig:example1-R}. Figure~\ref{fig:example1-definitions} illustrates the relevant sets of the example program, where black lines indicate the follower's feasible set $\Y(x)$ and orange dashed lines the leader's feasible set. Black dots correspond to follower feasible responses, where the green ones are optimal. Point $(3, 3)$ is not bilevel-feasible as it violates leader's constraint $2x + 2y \leq 11$. Therefore, the remaining three points $(0, 1)$, $(1, 1)$, and $(2, 0)$ form the bilevel-feasible set $\F$. An optimal solution is obtained by optimizing the leader's objective function over $\F$, which corresponds to $(1, 1)$ with an objective value $-5$. The HPR, illustrated in Figure~\ref{fig:example1-HPR}, achieves optimal solution $(2, 3)$ with objective value $-14$. 
    
    The associated binary expansion introduces new variables $x_1, x_2, y_1, y_2 \in \bin$ which satisfy $x = x_1 + 2x_2$ and $y = y_1 + 2y_2$, yielding the binary bilevel program
    \begin{alignat*}{2}
        \min & & \hspace{.2cm} & -x_1 - 2x_2 - 4y_1 - 8y_2 \\
        \st & & & 2x_1 + 4x_2 - 2y_1 - 4y_2 \leq 3 \\
        & & & 2x_1 + 4x_2 + 2y_1 + 4y_2 \leq 11 \\
        & & & x_1, x_2 \in \bin \\
        & & & y_1 + 2y_2 \in \responseSet(x_1 + 2x_2). 
    \end{alignat*}

    Its optimal solution is $(x_1, x_2, y_1, y_2) = (1, 0, 1, 0)$, corresponding to $x = 1$, $y = 1$, with objective value $-5$. We keep this binary representation throughout this work for illustrative purposes. \qedwhite
\end{example}

\begin{figure}[t]
    \centering

    \begin{subfigure}[t]{0.31\textwidth}
        \centering
        \begin{minipage}[c][3.2cm][c]{\linewidth}
            \centering
            \small
            \begin{adjustbox}{max width=\linewidth,center}
                $\displaystyle
                \begin{aligned}
                    \min_{x,y}\quad & -x - 4y \\
                    \st\quad        & 2x - 2y \leq 3 \\
                                    & 2x + 2y \leq 11 \\
                                    & x \in \Z \\
                                    & y \in \responseSet(x)
                \end{aligned}$
            \end{adjustbox}
        \end{minipage}
        \caption{Bilevel model}
        \label{fig:example1-model}
    \end{subfigure}
    \hfill
    \begin{subfigure}[t]{0.34\textwidth}
        \centering
        \begin{minipage}[c][3.2cm][c]{\linewidth}
            \centering
            \small
            \begin{adjustbox}{max width=\linewidth,center}
                $\displaystyle
                \begin{aligned}
                    \argmin_y \quad    & y \\
                    \st\quad        & 3 - 2x \leq 4y \leq 18 - 2x \\
                                    & -9 + 4x \leq y \leq 3 + 4x \\
                                    & 4y \leq 12 + x \\
                                    & y \in \Z_+
                \end{aligned}$
            \end{adjustbox}
        \end{minipage}
        \caption{Follower's optimal response set $\responseSet(x)$}
        \label{fig:example1-R}
    \end{subfigure}
    \hfill
    \begin{subfigure}[t]{0.31\textwidth}
        \centering
        \begin{minipage}[c][3.2cm][c]{\linewidth}
            \centering
            \small
            \begin{adjustbox}{max width=\linewidth,center}
                $\displaystyle
                \begin{aligned}
                    \min_{x,y} \quad & -x - 4y \\
                    \st\quad         & 2x - 2y \leq 3 \\
                                     & 2x + 2y \leq 11 \\
                                     & x \in \Z \\
                                     & y \in \Y(x)
                \end{aligned}$
            \end{adjustbox}
        \end{minipage}
        \caption{High-point relaxation}
        \label{fig:example1-HPR}
    \end{subfigure}

    \vspace{.3cm}
    \caption{Example integer bilevel programming model (a), the follower's optimal response set (b), and the associated high-point relaxation (c).}
\end{figure}

The graph of function $\phi$ corresponds to the set~$\gph \phi = \{(x, z) \in \dom \Y \times \R \cup \{ +\infty \} : z = \phi(x) \}$, and the hypograph of $\phi$ corresponds to the set~$\hyp \phi = \{(x, z) \in \dom \Y \times \R \cup \{ +\infty \} : z \leq \phi(x) \}$. We can now write program \eqref{model:BBP} as
\begin{equation} \label{model:BBP-HypReformulation}
    \min \left\{ c_L^\top x + c_F^\top y : (x, y) \in \H,\ ( x, d^\top y ) \in \hyp \phi \right\}.
\end{equation}

Figure \ref{fig:example1-hypPhi} illustrates $\hyp \phi$ for the follower's value function of the program in Example~\ref{ex:example1}.

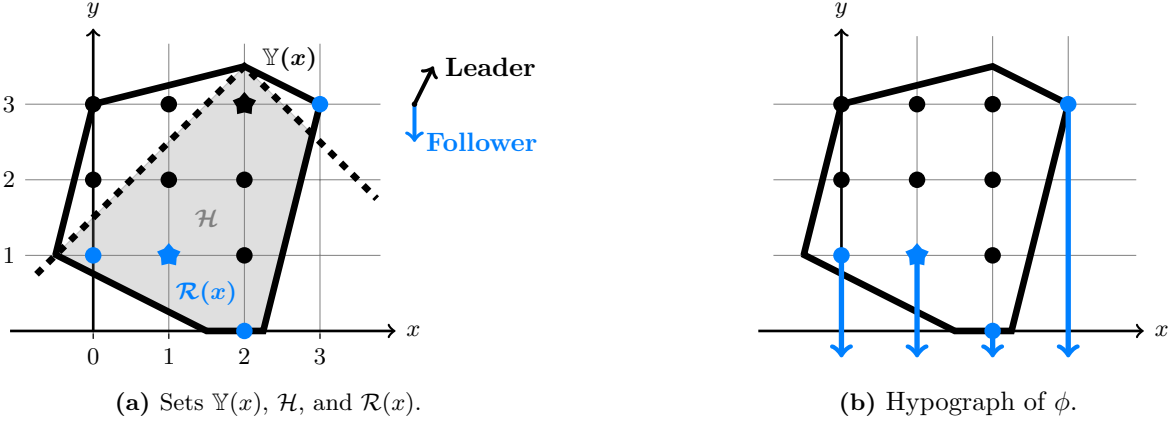
\begin{figure}[t]
    \centering
    \begin{subfigure}[t]{0.5\textwidth}
        \centering
        \begin{tikzpicture}[scale=1.00, every node/.style={font=\footnotesize}]
            \draw[very thin, color=gray] (-.9, -.1) grid (3.8, 3.8); 
        
            \draw[->, line width=1] (-1.1,0) -- (4, 0) node[right] {$x$};
            \draw[->, line width=1] (0, 0) -- (0, 4) node[above] {$y$};
            \foreach \x in {0, 1, 2, 3}
                \draw (\x cm, -1pt) -- (\x cm,-3pt) node[anchor=north] {$\x$};
            \foreach \y in {1, 2, 3}
                \node[anchor=east] at (-.9, \y cm) {$\y$};
        
            \draw[fill=gray, opacity=0.25] (-0.5, 1) -- (2, 3.5) -- (2.9, 2.6) -- (2.25, 0) -- (1.5, 0) -- (-0.5, 1);
        
            \draw[->, thick, myblue, line width=1.5] (4.25, 3) -- (4.25, 2.5) node[right] {\color{myblue}\small \textbf{Follower}};
            \draw[->, thick, line width=1.5] (4.25, 3) -- (4.5, 3.5) node[right] {\small \textbf{Leader}};
            \node[fill=black, circle, scale=.2] at (4.25, 3) {};
        
            \draw[line width=2.5] (-0.5, 1) -- (0, 3) -- (2, 3.5) -- (3, 3) -- (2.25, 0) -- (1.5, 0) -- (-0.52, 1.02);
            \draw[dashed, line width=2.5] (-0.75, 0.75) -- (2, 3.5) -- (3.75, 1.75);
            
            \node[fill=black, circle, scale=.6] at (0, 3) {};
            \node[fill=black, circle, scale=.6] at (0, 2) {};
            \node[fill=black, circle, scale=.6] at (1, 2) {};
            \node[fill=black, circle, scale=.6] at (1, 3) {};
            \node[fill=black, circle, scale=.6] at (2, 1) {};
            \node[fill=black, circle, scale=.6] at (2, 2) {};    
            \node[fill=black, star, scale=.6] at (2, 3) {};
        
            \node[fill=myblue, circle, scale=.6] at (0, 1) {};
            \node[fill=myblue, star, scale=.6] at (1, 1) {};
            \node[fill=myblue, circle, scale=.6] at (2, 0) {};    
            \node[fill=myblue, circle, scale=.6] at (3, 3) {};
        
            \node[] at (2.6, 3.6) {$\bm{\Y(x)}$};
            \node[] at (1.5, 1.5) {\color{gray} $\bm{\H}$};
            \node[] at (1.5, 0.5) {\color{myblue} $\bm{\responseSet(x)}$};
        \end{tikzpicture}
        \caption{Sets $\Y(x)$, $\H$, and $\responseSet(x)$.}
        \label{fig:example1-definitions}
    \end{subfigure}%
    ~\hspace{.7cm}
    \begin{subfigure}[t]{0.5\textwidth}
        \centering
        \begin{tikzpicture}[scale=1.00, every node/.style={font=\footnotesize}]
            \draw[very thin, color=gray] (-.9, -.1) grid (3.9, 3.9); 
        
            \draw[->, line width=1] (-1.1,0) -- (4, 0) node[right] {$x$};
            \draw[->, line width=1] (0, 0) -- (0, 4) node[above] {$y$};
            \foreach \x in {0, 1, 2, 3}
                \draw (\x cm, -1pt) -- (\x cm,-3pt) node[anchor=north] {\color{white}$\x$};
            \foreach \y in {1, 2, 3}
                \node[anchor=east] at (-.9, \y cm) {\color{white}$\y$};
        
        
            \draw[line width=2.5] (-0.5, 1) -- (0, 3) -- (2, 3.5) -- (3, 3) -- (2.25, 0) -- (1.5, 0) -- (-0.52, 1.02);
            
            \node[fill=black, circle, scale=.6] at (0, 3) {};
            \node[fill=black, circle, scale=.6] at (0, 2) {};
            \node[fill=black, circle, scale=.6] at (1, 2) {};
            \node[fill=black, circle, scale=.6] at (1, 3) {};
            \node[fill=black, circle, scale=.6] at (2, 1) {};
            \node[fill=black, circle, scale=.6] at (2, 2) {};    
            \node[fill=black, circle, scale=.6] at (2, 3) {};

            \draw[->, line width=2, myblue] (0, 1) -- (0, -0.35);
            \draw[->, line width=2, myblue] (1, 1) -- (1, -0.35);
            \draw[->, line width=2, myblue] (2, 0) -- (2, -0.35);
            \draw[->, line width=2, myblue] (3, 3) -- (3, -0.35);
        
            \node[fill=myblue, circle, scale=.6] at (0, 1) {};
            \node[fill=myblue, star, scale=.6] at (1, 1) {};
            \node[fill=myblue, circle, scale=.6] at (2, 0) {};    
            \node[fill=myblue, circle, scale=.6] at (3, 3) {};
        
        \end{tikzpicture}
        \caption{\small Hypograph of $\phi$.}
        \label{fig:example1-hypPhi}
    \end{subfigure}
    \vspace{.2cm}
    \caption{Sets for the example integer bilevel program. The leader's feasible set is depicted by the region enclosed by the dashed lines. The bilevel-feasible set $\F$ consists of pairs $(x,y)$ in this region satisfying $y \in \responseSet(x)$. In figure (b), the blue downward rays, including their upper endpoints, represent the hypograph of $\phi$, containing all points $(x,d^\top y)$ with $d^\top y \le \phi(x)$; here, $d^\top y=y$.}
    \label{fig:example1}
\end{figure}

\medskip

\textbf{Computing the value function via network flows.}
To compute the follower's value function via linear programming, \cite{LBC24} propose a network flow-based reformulation. The authors first show how to compile a layered directed acyclic network~$\D_{\gph \phi} = (\N, \A, l, \nu)$, with root node $r$ and sink node $t$, that encodes all points $(x, z)$ in $\gph \phi$ through $r$-$t$ paths. Here, $\N$ is a set of nodes, $\A$ is a set of arcs, $l$ is an arc-length vector, and $\nu$ is an arc-value vector. Furthermore, the arc set $\A$ is partitioned as $\A = \A^0_L \dotcup \A^1_L \dotcup \A_F$, where $\A^0_L$ and $\A^1_L$ are the set of arcs with value $0$ and $1$, respectively, encoding $x$ values, and $\A_F$ is the set of arcs encoding $z$ values. Moreover, arc lengths in $l$ are set to $0$ for arcs in $\A_L = \A^1_L \cup \A^0_L$, and to the evaluation of $\phi$ for arcs in $\A_F$. For any arc $a \in \A_L$, we denote as $\idx(a)$ the component of $x$ associated with $a$. This way, for any point $(\xhat, \zhat) \in \gph \phi$, there is exactly one $r$-$t$ path $\mathcal{P} = \{a_1, \ldots, a_{n_L + 1}\}$ such that $\xhat_{\idx(a)} = \nu_a$ for all $a \in \mathcal{P} \setminus \{a_{n_L + 1}\}$ and $\zhat = \sum_{a \in \mathcal{P}} l_a = l_{a_{n_L + 1}}$. Figure \ref{fig:exampleNetwork} depicts the network that encodes the graph of the follower's value function associated with the binary expansion of the problem in the running example~\ref{ex:example1}. Given $\D_{\gph \phi}$, we have that
\begin{subequations} \label{model:LBC-graphReformulation}
    \begin{alignat}{3}
        \min & & \hspace{.2cm} & c_L^\top x + c_F^\top y \\
        \st & & & (x, y) \in \H \\
        & & & \sum_{a \in \delta^+(u)} w_a - \sum_{a \in \delta^-(u)} w_a = \mathbf{1}(u = r) - \mathbf{1}(u = t) & \hspace{.5cm} & u \in \N \label{constr:LBC-graphReformulation-flow} \\
        & & & x_j = \sum_{a \in \A_L: \idx(a) = j} \nu_a w_a & & j \in [n_L] \label{constr:LBC-graphReformulation-x} \\
        & & & d^\top y \leq l^\top w \label{constr:LBC-graphReformulation-z} \\
        & & & w \geq 0. \label{constr:LBC-graphReformulation-w}
    \end{alignat}
\end{subequations}
is an exact reformulation of program \eqref{model:BBP-HypReformulation}, where $\mathbf{1}(\cdot)$ denotes the indicator function. Note that directly solving program~\eqref{model:LBC-graphReformulation} with an off-the-shelf MIP solver may quickly become intractable as~$\D_{\gph \phi}$ grows. To address this issue, \cite{LBC24} propose relaxations obtained from replacing~$\D_{\gph \phi}$ in \eqref{model:LBC-graphReformulation} with an approximate value network. This corresponds to another network $\D_{\gph \widehat{\phi}}$ that encodes points in~$\gph \widehat{\phi}(x)$, for some upper bounding function~$\widehat{\phi} \geq \phi$. This approach leads to a trade-off between network size and bound quality: in general, larger networks provide stronger bounds, as they encode the graph of more accurate approximations of $\phi$. Thus, as long as compilation remains tractable, it is desirable to work with larger networks.

\begin{figure}[t]
    \centering
    \begin{tikzpicture}[scale=0.75, ->, >=stealth', shorten >=1pt, auto, node distance=3cm, thick, 
            main node/.style={circle, draw, minimum size=4mm, scale=.75},
            special node/.style={circle, draw, font=\sffamily\bfseries, minimum size=4mm, scale=.85}
        ]
        \node[special node] (r) at (0, 0) {\small r};
        
        \node[main node] (1) at (-1.5, -2) {};
        \node[main node] (2) at (1.5, -2) {};
        \node[main node] (3) at (-2, -4) {};
        \node[main node] (4) at (0, -4) {};
        \node[main node] (5) at (2, -4) {};
    
        \node[special node] (t) at (0, -6.5) {\small t};
    
        \node[] (x1) at (3, -1) {\small $\bm{x_1}$};
        \node[] (x2) at (3, -3) {\small $\bm{x_2}$};
        \node[] (z) at (3, -5) {\small \color{myblue} $\bm{z}$};
    
        \path[every node/.style={font=\sffamily\footnotesize,
            fill=white, inner sep=1pt}]
            (r) edge [dashed, line width=1.5] (1)
                edge [line width=1.5] (2)
            (1) edge [line width=1.5] (3)
                edge [dashed, line width=1.5] (4)
            (2) edge [dashed, line width=1.5] (4)
                edge [line width=1.5] (5)
            (3) edge [myblue, line width=1.5] node [pos=0.5, anchor=center, minimum size=5mm] {\small $\bm{0}$} (t)
            (4) edge [myblue, line width=1.5] node [pos=0.5, anchor=center, minimum size=5mm] {\small $\bm{1}$} (t)
            (5) edge [myblue, line width=1.5] node [pos=0.5, anchor=center, minimum size=5mm] {\small $\bm{3}$} (t);
    \end{tikzpicture}
    \vspace{.2cm}
    \caption{Network $\D_{\gph \phi}$ used to compute the follower's value function of the running example, using the binary representation $x = x_1 + 2x_2$. Black arcs are associated with $x$ values and blue arcs with $z$ values. Dashed arcs represent 0-arcs and solid arcs represent 1-arcs. Arc labels represent lengths.}
    \label{fig:exampleNetwork}
\end{figure}
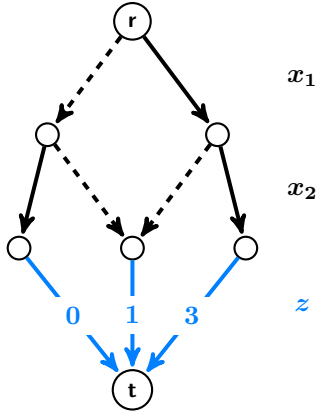

\section{Network flow-based reformulations} \label{sec:new-reformulations}
We introduce two alternative network flow-based reformulations of \eqref{model:hypConvReformulation} based on $\D_{\gph \phi}$. Both reformulations can be systematically outer approximated via linear inequalities obtained from $\D_{\gph \phi}$, which allows to use the network as an oracle and, thus, address the tractability issues of directly solving program \eqref{model:LBC-graphReformulation}. First, we state a simple yet important result that serves as the cornerstone of our reformulations. We refer to $\conv \X$ as the convex hull of set $\X$.

\begin{lemma} \label{thm:hypConvRelaxation}
    $\hyp \phi = \conv(\hyp \phi) \cap (\bin^{n_L} \times \R)$
\end{lemma}
\proof{Proof.}
    Let $\mathcal{S} := \hyp \phi$ and $\widebar{\mathcal{S}} := \conv(\hyp \phi) \cap (\bin^{n_L} \times \R)$. We want to prove $\mathcal{S} = \widebar{\mathcal{S}}$. We know~$\mathcal{S} \subseteq \widebar{\mathcal{S}}$, so we show~$\mathcal{S} \supseteq \widebar{\mathcal{S}}$. Let $(\xbar, \zbar) \in \widebar{\mathcal{S}}$ and consider a finite set $\X \subseteq \hyp \phi$ such that $(\xbar, \zbar) = \sum_{(x, z) \in \X} \lambda_{(x, z)} (x, z)$ for some vector $\lambda \in \R^{|\X|}_+$. As $\xbar \in \bin^{n_L}$, it is an extreme point of the box $[0, 1]^{n_L}$ and, hence, $x = \xbar$ for all $(x, z) \in \X$. This way, we obtain $\zbar = \sum_{(x, z) \in \X} \lambda_{(x, z)} z \leq \sum_{(x, z) \in \X} \lambda_{(x, z)} \phi(x) = \phi(\xbar)$. Thus, $(\xbar, \zbar) \in \hyp \phi$, concluding the proof. 
        \hfill \Halmos
\endproof

As $x \in \bin^{n_L}$ is implied by $(x, y) \in \H$, Lemma~\ref{thm:hypConvRelaxation} provides a relaxation of program \eqref{model:BBP-HypReformulation} which is tight at binary $x$ values.

\begin{corollary} \label{thm:hypConvRelaxationCorollary}
    Program
    \begin{subequations} \label{model:hypConvReformulation}
        \begin{alignat}{3}
            \min & & \hspace{.2cm} & c_L^\top x + c_F^\top y \\
            \st & & & (x, y) \in \H \\
            & & & (x, d^\top y) \in \conv( \hyp \phi ).
        \end{alignat}
    \end{subequations}
    is an exact reformulation of program~\eqref{model:BBP-HypReformulation}.
\end{corollary}

Observe that, contrary to $\hyp \phi$, the set $\conv(\hyp \phi)$ is polyhedral, as $x$ can take finitely many values. This allows to build outer approximations via linear inequalities for program~\eqref{model:hypConvReformulation}. In particular, constraints~\eqref{constr:LBC-graphReformulation-flow}-\eqref{constr:LBC-graphReformulation-w} from the network flow-based reformulation proposed by \cite{LBC24} model $(x, d^\top y) \in \conv(\hyp \phi)$ on an extended variable space, where flow variables $w$ act as convex multipliers. Moreover, if $\D_{\gph \phi}$ is replaced with an approximate value network $\D_{\gph \widehat{\phi}}$ in program \eqref{model:LBC-graphReformulation}, then via Corollary~\ref{thm:hypConvRelaxationCorollary}, we can conclude that the resulting relaxation represents the following program 
\begin{center}
$\begin{array}{ll}
    \min &  c_L^\top x + c_F^\top y \\
    \st & (x, y) \in \H \\
     & (x, d^\top y) \in \conv( \hyp \widehat{\phi} ) \supseteq \conv (\hyp \phi)
\end{array}
$
\end{center}
\medskip
but in an extended variable space. 

%

As mentioned in Section~\ref{sec:preliminaries}, larger networks are useful since they provide better bounds on the optimal value of program~\eqref{model:BBP-HypReformulation} by yielding tighter outer approximations of $\conv(\hyp \phi)$. In what follows, we propose two alternative network flow-based reformulations of program \eqref{model:hypConvReformulation} that, contrary to \eqref{model:LBC-graphReformulation}, do not require embedding $\D_{\gph \phi}$. These reformulations remain in the original variable space and dynamically refine the outer approximation of $\conv(\hyp \phi)$ via cuts, allowing for improved scalability.

\begin{remark}
    Lemma \ref{thm:hypConvRelaxation} requires $x$ to be binary; Figure \ref{fig:example1-hypPhi} serves as a counterexample for the general integer case, since $(x, z) = (2, 2)$ belongs to $\conv(\hyp \phi) \cap (\Z^{n_L} \times \R)$, but not to $\hyp \phi$.
\end{remark}

\medskip

\subsection{Optimality-based reformulation} \label{sec:optCuts}
For any vector $\xhat \in \dom\Y \cap \bin^{n_L}$, we have $\phi(\xhat) = \max\{z: (\xhat, z) \in \conv(\hyp \phi) \}$. Therefore, given $\D_{\gph \phi} = (\N, \A, l, \nu)$, we obtain
\begin{subequations} \label{model:phi-optReformulation}
    \begin{alignat}{4}
        \phi(\xhat) & = \ & \max & & \hspace{.2cm} & l^\top w \\
        & & \st & & & \sum_{a \in \delta^+(u)} w_a - \sum_{a \in \delta^-(u)} w_a = \mathbf{1}(u = r) - \mathbf{1}(u = t) & \hspace{.5cm} & u \in \N \\
        & & & & & \sum_{a \in \A_L: \idx(a) = j} \nu_a w_a = \xhat_j & & j \in [n_L] \\
        & & & & & w \geq 0.
    \end{alignat}
\end{subequations}

By strong duality, we can reformulate $\eqref{model:BBP-valueFunctionRef}$ as
$$
    \min_{x, y} \left\{ c_L^\top x + c_F^\top y:\ \begin{array}{ll}
        (x, y) \in \H, \vspace{.1cm} \\
        d^\top y \leq \underset{\pi, \lambda}{\min} \left\{ \pi_r + x^\top \lambda:\ (\pi, \lambda) \in \mathcal{P}(\D_{\gph \phi}) \right\} 
    \end{array} \right\},
$$
where
$$
    \mathcal{P}(\D_{\gph \phi}) := \left\{(\pi, \lambda) \in \R^{\left| \N \right|} \times \R^{n_L}:\ \begin{array}{ll}
        \pi_u - \pi_v + \nu_{uv}\lambda_{\idx(u, v)} \geq 0 & (u, v) \in \A_L, \vspace{.1cm} \\
        \pi_u - \pi_v \geq l_{uv} & (u, v) \in \A_F, \vspace{.1cm} \\
        \pi_t = 0
    \end{array}\right\}
$$
corresponds to the feasible solution set of the dual of program \eqref{model:phi-optReformulation}. This way, we obtain the exact reformulation
\begin{subequations} \label{model:optReformulation}
    \begin{alignat}{3}
        \min & & \hspace{.2cm} & c_L^\top x + c_F^\top y \\
        \st & & & (x, y) \in \H \\
        & & & d^\top y \leq \pi_r + \lambda^\top x & \hspace{.5cm} & (\pi, \lambda) \in \mathcal{P}( \D_{\gph \phi} ). \label{constr:optReformulation-valueFunction}
    \end{alignat}
\end{subequations}

\begin{proposition} \label{thm:optimalityReformulationEquivalence}
    Program \eqref{model:optReformulation} is an exact reformulation of \eqref{model:hypConvReformulation}.
\end{proposition}
\proof{Proof.}
    As both programs consider $(x, y) \in \H$, we focus on comparing $(x, d^\top y) \in \conv(\hyp \phi)$ with constraints \eqref{constr:optReformulation-valueFunction}. We first prove that any solution $(\xhat, \yhat)$ feasible to program \eqref{model:optReformulation} is also feasible to program \eqref{model:hypConvReformulation}. Equivalently, we show that if~$(\xhat, d^\top\yhat) \not\in \conv(\hyp \phi)$, then there exists a dual vector $(\pi, \lambda)$ for which $(\xhat, \yhat)$ violates \eqref{constr:optReformulation-valueFunction}. Since $(\xhat, \yhat) \in \H$, then we have $d^\top \yhat > \phi(\xhat)$. Due to strong duality, there exists a vector $(\pihat, \lambdahat) \in \argmin \{ \pi_r + \xhat^\top \lambda:\ (\pi, \lambda) \in \mathcal{P}( \D_{\gph \phi} ) \}$ such that~$\phi(\xhat) = \pihat_r + \lambdahat^\top \xhat$. Hence, we get $d^\top \yhat > \phi(\xhat) = \pihat_r + \lambdahat^\top \xhat$, so $(\xhat, \yhat)$ is not feasible to program \eqref{model:optReformulation}.
    
    We now prove that any solution $(\xhat, \yhat)$ feasible to program \eqref{model:hypConvReformulation} is also feasible to program \eqref{model:optReformulation}. By weak duality, we have $d^\top \yhat \leq \phi(\xhat) \leq \min \{ \pi_r + \xhat^\top \lambda : (\pi, \lambda) \in \mathcal{P}( \D_{\gph \phi} ) \}$. Hence, $(\xhat, \yhat)$ is feasible to program \eqref{model:optReformulation}. \hfill \Halmos
\endproof

In this work, we refer to inequalities~\eqref{constr:optReformulation-valueFunction} as \textit{optimality cuts}. Observe that optimality cuts can be dynamically separated by solving the linear program~$\min \{ \pi_r + \lambda^\top x:\ (\pi, \lambda) \in \mathcal{P}( \D_{\gph \phi} ) \}$. However, we can achieve a more efficient algorithm by exploiting $\D_{\gph \phi}$. By adding (redundant) constraints
$$
    \sum_{a \in \A_L: \idx(a) = j} (1 - \nu_a) w_a = 1 - \xhat_j \qquad \forall j \in [n_L]
$$
to program \eqref{model:phi-optReformulation} and then taking its dual, it becomes clear that the separation problem $\min\{ \pi_r + \xhat^\top \lambda: (\pi, \lambda) \in \mathcal{P}(\D_{\gph \phi}) \}$ is equivalent to
\begin{subequations} \label{model:optReformualation-separation}
    \begin{alignat}{3}
        \min_{\pi, \lambda^1, \lambda^0} & & \hspace{.2cm} & \pi_r + \sum_{j \in [n_L]} \xhat_j \lambda^1_j + (1 - \xhat_j) \lambda^0_j \\
        \st & & & \pi_u - \pi_v + \nu_{uv}\lambda^1_{\idx(u, v)} + (1 - \nu_{uv})\lambda^0_{\idx(u, v)} \geq 0 & \hspace{.5cm} & (u, v) \in \A_L \label{constr:optReformulation-separation-potential} \\
        & & & \pi_u - \pi_v \geq l_{uv} && (u, v) \in \A_F \\
        & & & \pi_t = 0.
    \end{alignat}
\end{subequations}

This equivalent program allows to fix the value of $\pi_r$ without loss of optimality and, in turn, use a modification of the approach proposed in \cite{lozano2018stochastic} to derive a dual optimal solution.

\begin{proposition}[{\cite[Proposition 3]{vasquez2025single}}] \label{thm:optReformulation-separationOptSol}
    Program \eqref{model:optReformualation-separation} admits an optimal solution $(\pi^\ast, \lambda^{1^\ast}, \lambda^{0^\ast})$ satisfying~$\pi_r^\ast = \phi(\xhat)$ and~$\xhat_j \lambda^{1^\ast}_j = (1 - \xhat_j) \lambda^{0^\ast}_j = 0$ for all~$j \in [n_L]$.
\end{proposition} 

A natural optimal solution $(\widehat{\pi}, \widehat{\lambda}^1, \widehat{\lambda}^0)$ to program \eqref{model:optReformualation-separation} can be obtained by setting $\widehat{\pi}_r = \phi(\xhat)$, $\widehat{\pi}_t = 0$, and $\widehat{\lambda}^1_j = M (1 - \xhat_j)$, $\widehat{\lambda}^0_j = M \xhat_j$ for all $j \in [n_L]$, where $M$ is a large enough constant. The resulting inequality corresponds to the traditional ``no-good'' cut
$$
    -M \left[ \sum_{j \in [n_L]} (1 - \xhat_j) x_j + \xhat_j (1 - x_j) \right] + d^\top y \leq \phi(\xhat)
$$
which is tight at $x = \xhat$, but may be relatively weak for different values. To obtain a tighter inequality, we now describe a simple procedure that leverages $\D_{\gph \phi}$ to compute an alternative dual optimal solution satisfying the conditions from Proposition \ref{thm:optReformulation-separationOptSol}. We start by setting $\widehat{\pi}_r = \phi(\xhat)$ and $\widehat{\pi}_t = 0$. Then, for each arc $(u, t) \in \A_F$, we set $\widehat{\pi}_u = l_{ut}$, and for each arc $(u, v) \in \A_L$ such that $u \neq r$, we set $\widehat{\pi}_u = \mathbf{1}(\nu_{uv} = \xhat_{\idx(uv)})\widehat{\pi}_v$. Note that, as each node has at most two outgoing arcs with different value in $\nu$, the value assignment is consistent. Furthermore, it guarantees that if there exists a (unique) $u$-$t$ path encoding a subvector of $\xhat$, then $\widehat{\pi}_u$ is equal to its length. 
%
%
Finally, we build $\widehat{\lambda}^1$ and $\widehat{\lambda}^0$ by first setting $\widehat{\lambda}^1_j = 0$ for all $j \in [n_L]$ such that $\xhat_j = 1$, and $\widehat{\lambda}^0_j = \xhat_j$ for all $j \in [n_L]$ such that $\xhat_j = 0$. Then, for the remaining components, we set 
\begin{alignat*}{2}
    & \widehat{\lambda}^1_j = \max_{(u, v) \in \A_L: \idx(u, v) = j} \left\{ \pi_v - \pi_u \right\} & \hspace{.5cm} & \forall j \in [n_L]: \xhat_j = 0, \\
    & \widehat{\lambda}^0_j = \max_{(u, v) \in \A_L: \idx(u, v) = j} \left\{ \pi_v - \pi_u \right\} && \forall j \in [n_L]: \xhat_j = 1
\end{alignat*}
    
This way, we obtain the valid inequality $(\widehat{\lambda}^0 - \widehat{\lambda}^1)^\top x + d^\top y \leq \widehat{\pi}_r + \bm{1}^\top \widehat{\lambda}^0$, which is also tight at $x = \xhat$ due to optimality of $(\widehat{\pi}, \widehat{\lambda}^1, \widehat{\lambda}^0)$.

\medskip

\subsection{Feasibility-based reformulation}
Consider a point $(x, z) \in \bin^{n_L} \times \R$. Given a network $\D_{\gph \phi} = (\N, \A, l, \nu)$, the program
\begin{subequations} \label{model:feasibilityProgram}
    \begin{alignat}{3}
        \min_w & & \hspace{.2cm} & 0 \\
        \st & & & \sum_{a \in \delta^+(u)} w_a - \sum_{a \in \delta^-(u)} w_a = \mathbf{1}(u = r) - \mathbf{1}(u = t) & \hspace{.5cm} & u \in \N \\
        & & & \sum_{a \in \A_L: \idx(a) = j} \nu_a w_a = x_j & & j \in [n_L] \label{constr:feasibilityProgram-xValue} \\
        & & & z \leq \sum_{a \in \A_F} l_a w_a \label{constr:feasibilityProgram-zValue} \\
        & & & w \geq 0
    \end{alignat}
\end{subequations}
is feasible if and only if $(x, z) \in \conv(\hyp \phi)$. Hence, for any point $(\xhat, \zhat) \not\in \conv(\hyp \phi)$, the dual of program \eqref{model:feasibilityProgram}, corresponding to
\begin{subequations} \label{model:feasibilityProgramDual}
    \begin{alignat}{3}
        \max_{\pi, \lambda, \gamma} & & \hspace{.2cm} & \pi_r + \xhat^\top \lambda + \zhat \gamma \\
        \st &&& \pi_u - \pi_v + \nu_{uv} \lambda_{\idx(u, v)} \leq 0 & \hspace{.5cm} & (u, v) \in \A_L \\
        &&& \pi_u - \pi_v + l_{uv}\gamma \leq 0 & & (u, v) \in \A_F \\
        &&& \pi_t = 0 \\
        &&& \gamma \geq 0
    \end{alignat}
\end{subequations}
is unbounded (the feasible region contains the origin). This way, given a feasible dual solution~$(\widehat{\pi}, \widehat{\lambda}, \widehat{\gamma})$, the inequality~$\widehat{\pi}_r + \widehat{\lambda}^\top x + \widehat{\gamma} z \leq 0$ is valid for~$\conv(\hyp \phi)$ and violated by~$(\xhat, \zhat)$. This leads to the exact reformulation of program \eqref{model:BBP-valueFunctionRef}
\begin{subequations} \label{model:feasReformulation}
    \begin{alignat}{3}
        \min & & \hspace{.2cm} & c_L^\top x + c_F^\top y \\
        \st & & & (x, y) \in \H \\
        & & & \pi_r + \lambda^\top x + \gamma d^\top y \leq 0 & \hspace{.5cm} & (\pi, \lambda, \gamma) \in \mathcal{Q}( \D_{\gph \phi} ). \label{constr:feasReformulation-valueFunction}
    \end{alignat}
\end{subequations}
where
$$
    \mathcal{Q}(\D_{\gph \phi}) = \left\{(\pi, \lambda, \gamma) \in \R^{\left| \N \right|} \times \R^{n_L} \times \R_+:\ \begin{array}{ll}
        \pi_u - \pi_v + \nu_{uv}\lambda_{\idx(u, v)} \leq 0 & (u, v) \in \A_L, \\
        \pi_u - \pi_v + l_{uv}\gamma \leq 0 & (u, v) \in \A_F, \\
        \pi_t = 0
    \end{array}\right\}
$$
is the cone defining the feasible region of the dual of program \eqref{model:feasibilityProgram}.

\begin{proposition} \label{thm:feasibilityReformulationEquivalence}
    Program \eqref{model:feasReformulation} is an exact reformulation of \eqref{model:hypConvReformulation}.
\end{proposition}
\proof{Proof.}
    As both programs consider $(x, y) \in \H$, we focus on comparing $(x, d^\top y) \in \conv(\hyp \phi)$ with constraints \eqref{constr:feasReformulation-valueFunction}. We first prove that any solution $(\xhat, \yhat)$ feasible to program \eqref{model:feasReformulation} is also feasible to program \eqref{model:hypConvReformulation}. Equivalently, if $(\xhat, d^\top \yhat) \not\in \conv(\hyp \phi)$, then there exists a ray~$(\pihat, \lambdahat, \gammahat) \in \mathcal{Q}( \D_{\gph \phi} )$ such that~$\pihat_r + \lambdahat^\top \xhat + \gammahat d^\top \yhat > 0$. Hence,~$(\xhat, \yhat)$ is not feasible to program \eqref{model:feasReformulation}.
    
    We now prove that any solution $(\xhat, \yhat)$ feasible to program \eqref{model:hypConvReformulation} is also feasible to program \eqref{model:feasReformulation}. 
    First, observe that
    \vspace{-.2cm}
    \[
        \proj_{\pi, \lambda} (\mathcal{Q}(\D_{\gph \phi}) \cap \{\gamma = t\}) = \begin{cases}
            -t \mathcal{P}(\D_{\gph \phi}) & t > 0, \\
            -\rec(\mathcal{P}(\D_{\gph \phi})) & t = 0,
        \end{cases}
    \]
    where $\rec(\mathcal{P}(\D_{\gph \phi}))$ denotes the recession cone of $\mathcal{P}(\D_{\gph \phi})$. For $t > 0$, we have that any ray $(\pihat, \lambdahat, \gammahat) \in \mathcal{Q}(\D_{\gph \phi})$ yields~$(-\pihat, -\lambdahat) \in \gammahat \mathcal{P}(\D_{\gph \phi})$. By weak duality, we have~$\min \{ \pi_r + \lambda^\top x: (\pi, \lambda) \in \gammahat\mathcal{P}( \D_{\gph \phi} ) \} \geq \gammahat \phi(\xhat)$. Hence, we obtain that~$-\pihat_r - \lambdahat^\top \xhat \geq \gammahat\phi(\xhat) \geq \gammahat d^\top \yhat$ holds for any ray $(\pihat, \lambdahat, \gammahat) \in \mathcal{Q}(\D_{\gph \phi})$ such that $\gammahat > 0$. For $t = 0$, since $\phi(\xhat)$ is finite, every recession direction $(\pi, \lambda)$ satisfies $\pi_r + \lambda^\top \xhat \geq 0$. Therefore, every ray $(\pihat, \lambdahat, 0) \in \mathcal{Q}(\D_{\gph \phi})$ satisfies $\pihat_r + \lambdahat^\top \xhat \leq 0$. Thus, $(\xhat, \yhat)$ is feasible to program \eqref{model:feasReformulation}. \hfill \Halmos
\endproof


We refer to inequalities~\eqref{constr:feasReformulation-valueFunction} as \textit{feasibility cuts}. The proof of Proposition \ref{thm:feasibilityReformulationEquivalence} also shows how feasibility cuts generalize optimality cuts. In particular, every feasibility cut with $\gamma > 0$ can be normalized to obtain an optimality cut, whereas feasibility cuts with $\gamma = 0$ are induced by negatives of recession directions of $P(\D_{\gph \phi})$. Observe that feasibility cuts can be separated by directly solving the dual of program~\eqref{model:feasibilityProgram}. Alternatively, we can use a projected subgradient approach, as proposed in \cite{davarnia2021outer}, to solve a Lagrangian dual of program~\eqref{model:feasibilityProgram}. Given $(\xhat, \zhat) \not\in \conv(\hyp \phi)$, $\mu^L \in \R^{n_L}$, and $\mu^F \geq 0$, define the Lagrangian function
\begin{alignat*}{3}
    \mathcal{L}(\mu^L, \mu^F) =  \min_{w \geq 0} & & \hspace{.2cm} & \sum_{j \in [n_L]} \mu_j^L \left( \xhat_j - \sum_{a \in \A^1_L: \idx(a) = j} w_a \right) + \mu^F \left( \zhat - \sum_{a \in \A_F} l_a w_a \right) \\
    \st & & & \sum_{a \in \delta^+(u)} w_a - \sum_{a \in \delta^-(u)} w_a = \mathbf{1}(u = r) - \mathbf{1}(u = t) & \hspace{.5cm} & u \in \N
\end{alignat*}
obtained by relaxing \eqref{constr:feasibilityProgram-xValue}-\eqref{constr:feasibilityProgram-zValue} from program \eqref{model:feasibilityProgram}. A supergradient of $\mathcal{L}$ is $\partial \mathcal{L}(\mu^L, \mu^F) = (\xhat, \zhat) - p$, where $p$ is a longest path on $\D_{\gph \phi}$ after redefining the arc length vector $l$ using the function 
$$
    l_a(\mu^L, \mu^F) = \begin{cases}
        \nu_a \mu^L_{\idx(a)} & \text{if } a \in \A_L, \\
        l_a\mu^F & \text{if } a \in \A_F,
    \end{cases}
$$
for all arc $a \in \A$.
Then, we can solve the Lagrangian dual $\max\{ \mathcal{L}(\mu^L, \mu^F) : \mu^F \geq 0 \}$ using a projected subgradient method. Alternatively, due to the definition of the hypograph, we can set $\mu^F = 1$ and use a subgradient method. Note that, due to weak duality, any dual vector $(\mu^L, \mu^F)$ yielding $\mathcal{L}(\mu^L, \mu^F) > 0$ is sufficient to build a feasibility cut violated by $(\xhat, \yhat)$.

Before finishing this Section, we make a brief comment about the dimension of the faces induced by valid inequalities.

\begin{remark}
    Let $\widetilde{\F} := \{ (x, y): x \in \dom \Y,\ y \in R(x)\}\) denote the set obtained from $\F$ after relaxing the leader constraints. If an inequality~$\alpha^\top x + \gamma z \leq \beta$, with $\gamma > 0$, induces a face~$F \subseteq \conv(\hyp \phi)$ such that~$\dim F = k$, then~$\alpha^\top x + \gamma d^\top y \leq \beta$ induces a face~$\widetilde{F} \subseteq \conv \widetilde{\F}$ such that~$\dim \widetilde{F} \geq k$. Indeed, since $\gamma > 0$, the points in $F$ belong to~$\conv(\gph \phi)$. Therefore, each of $k + 1$ affinely independent points in $F$ can be lifted to a point in $\widetilde{F}$. These lifted points remain affinely independent because their images under the map~$(x, y) \mapsto (x, d^\top y)$ are affinely independent. Note that this dimension bound does not necessarily extend to $\conv \F$, since the leader constraints may exclude some of the lifted points.
\end{remark}

In what follows, we show how to strengthen structured valid inequalities via tilting by leveraging an arc length-parametrized version of $\D_{\gph \phi}$.

\section{Strengthening cuts via tilting} \label{sec:cut-strengthening}
We next introduce a tilting procedure for structured valid inequalities inspired by the approaches in \cite{behle2007binary} and \cite{castro2022lifting}. This procedure computes \textit{tilting quantities}, corresponding to the maximum component-wise increase of the left-hand side that results in a strengthened valid inequality. We further explain how to exploit the combinatorial structure of $\D_{\gph \phi}$ to efficiently compute them.

\subsection{Coefficient Tilting} 
We focus on valid inequalities of the form 
\begin{equation} \label{eq:general-cut}
    \alpha^\top x + \gamma z \leq \beta,
\end{equation}
where $\alpha \in \R^{n_L}$, and $\gamma, \beta \in \R$. Moreover, let
$$
    F(\alpha, \gamma, \beta) = \left\{(x, z) \in \conv(\hyp \phi) : \alpha^\top \xbar + \gamma \zbar = \beta \right\}
$$ 
be the face of $\conv(\hyp \phi)$ induced by the cut defined by $(\alpha, \gamma, \beta)$. We want to know how much we can increase the coefficients of $\alpha$ while keeping the resulting inequality valid for $\conv(\hyp \phi)$. To do so, for each $j \in [n_L]$ and $\theta \in \bin$, we define the \textit{tilting quantity function} $t^\theta_j: \R^{n_L} \times \R \times \R \mapsto \R_+ \cup \{+ \infty\}$ as
\begin{equation} \label{eq:tilting-function}
    t_j^\theta(\alpha, \gamma, \beta) = \beta - \theta \alpha_j - \max_{x, z} \left\{ \gamma z + \sum_{i \in [n_L] \setminus \{j\}} \alpha_i x_i : (x, z) \in \conv(\hyp \phi),\ x_j = \theta \right\}.
\end{equation}

Functions $t^\theta_j$ compute the maximum increment $\alpha_j$ can achieve while keeping the tilted inequality
\begin{equation} \label{eq:tilted-cut}
    \left[ \alpha + \left( t^1_j(\alpha, \gamma, \beta) - t^0_j(\alpha, \gamma, \beta) \right) e^j \right]^\top x + \gamma z \leq \beta - t^0_j(\alpha, \gamma, \beta)
\end{equation}
valid for $\conv(\hyp \phi)$, where $e^j$ is the $j$-th canonical vector. Note that if~$t_j^\theta(\alpha, \gamma, \beta) = +\infty$, then $x_j = 1 - \theta$ throughout $\conv(\hyp \phi)$, and $x_j$ can be fixed. Hence, in what follows, we consider only components $j \in [n_L]$ for which both tilting quantities $t_j^0$ and $t_j^1$ are finite.

\begin{proposition} \label{thm:tilted-cut}
    Let $(\alpha, \gamma, \beta)$ define a valid inequality of the form \eqref{eq:general-cut}. Then, the tilted inequality \eqref{eq:tilted-cut} is valid for $\conv(\hyp \phi)$.
\end{proposition}
\proof{Proof.}
    Consider any vector $(\xhat, \zhat) \in \conv(\hyp \phi)$. By definition of $t_j^0$, we have
\begin{center}
$\begin{array}{ll}
      t^0_j(\alpha, \gamma, \beta) & = \beta - \max_{x, z} \left\{ \gamma z + \sum_{i \in [n_L] \setminus \{j\}} \alpha_i x_i : (x, z) \in \conv(\hyp \phi),\ x_j = 0 \right\} \\
        & \leq \beta - \sum_{i \in [n] \setminus \{j\}} \alpha_i \xhat_i - \gamma \zhat.
    \end{array}$
\end{center}
Therefore, if $\xhat_j = 0$, then $(\xhat, \zhat)$ satisfies \eqref{eq:tilted-cut}. Similarly, by definition of $t_j^1$, we have
    \begin{center}$
    \begin{array}{ll}
        t^1_j(\alpha, \gamma, \beta) & = \beta - \alpha_j - \max_{x, z} \left\{ \gamma z + \sum_{i \in [n_L] \setminus \{j\}} \alpha_i x_i : (x, z) \in \conv(\hyp \phi),\ x_j = 1 \right\} \\
        & \leq \beta - \alpha_j - \sum_{i \in [n] \setminus \{j\}} \alpha_i \xhat_i - \gamma \zhat.
    \end{array}$
    \end{center}
    By subtracting $t_j^0(\alpha, \gamma, \beta)$ on both sides of the inequality, we obtain that $(\xhat, \zhat)$ satisfies \eqref{eq:tilted-cut}. We conclude that \eqref{eq:tilted-cut} is valid for~$\conv(\hyp \phi)$. \hfill \Halmos
\endproof

We now use Theorem 2 in \cite{castro2022lifting} to show that tilted inequalities induce higher dimensional faces.

\begin{proposition} \label{thm:tilted-cut-dimension}
    Let $(\alpha, \gamma, \beta)$ define a valid inequality of the form \eqref{eq:general-cut}. Assume there exists $j \in [n_L]$ such that $t^\theta_j(\alpha, \gamma, \beta) > 0$ for some $\theta \in \bin$, and let $\widehat{\alpha} = \alpha + (t^1_j(\alpha, \gamma, \beta) - t^0_j(\alpha, \gamma, \beta)) e^j$, $\widehat{\beta} = \beta - t^0_j(\alpha, \gamma, \beta)$. Then, $\dim F(\widehat{\alpha}, \gamma, \widehat{\beta}) \geq \dim F(\alpha, \gamma, \beta) + 1$.
\end{proposition}
\proof{Proof.}
    We distinguish two cases: both tilting functions evaluate to a positive value, or only one of them does. If $t_j^0(\alpha, \gamma, \beta), t_j^1(\alpha, \gamma, \beta) > 0$, then there is no point $(x, z) \in F(\alpha, \gamma, \beta)$ such that $x_j = 0$ or $x_j = 1$, \ie, the inequality is not tight at any point in $\conv(\hyp \phi)$; otherwise, one of the two tilting functions would evaluate to zero. Therefore, after tilting, we obtain a new vector $(\widehat{\alpha}, \gamma, \widehat{\beta})$ such that $F(\widehat{\alpha}, \gamma, \widehat{\beta})$ contains a point $(x, z)$ with $x_j = 1$ (and another with $x_j = 0$). Hence, we only focus on the second case and assume $|F(\alpha, \gamma, \beta)| \geq 1$.
    
    If $t_j^\theta(\alpha, \gamma, \beta) = 0$ for $\theta = 0$ or $\theta = 1$, then any point $(x, z) \in F(\alpha, \gamma, \beta)$ satisfies $x_j = \theta$. By definition of $t_j^\theta$, after tilting, we obtain a new vector $(\widehat{\alpha}, \gamma, \widehat{\beta})$ such that $F(\widehat{\alpha}, \gamma, \widehat{\beta})$ contains a point $(x, z)$ satisfying $x_j = 1 - \theta$ which is affinely independent from all points in $F(\alpha, \gamma, \beta)$. Hence, $\dim F(\widehat{\alpha}, \gamma, \widehat{\beta}) \geq \dim F(\alpha, \gamma, \beta) + 1$. \hfill \Halmos
\endproof

As noted in \cite{castro2022lifting}, if $\conv(\hyp \phi)$ is full-dimensional, Proposition~\ref{thm:tilted-cut-dimension} indicates that whenever we tilt $n_L - \dim F(\alpha, \gamma, \beta)$ components of a valid inequality defined by $(\alpha, \gamma, \beta)$, the resulting inequality is facet-defining for $\conv(\hyp \phi)$. Nonetheless, it is possible to achieve facet-defining inequalities even if we tilt fewer components, \eg, if $\conv(\hyp \phi)$ is not full-dimensional or if tilting one component leads to an inequality inducing a face $F'$ satisfying $\dim F' > \dim F(\alpha, \gamma, \beta) + 1$.

\subsection{Computing tilting quantities with a network} \label{sec:tiltingFunctionsWithDDs}

In general, computing tilting quantity functions for binary linear programs poses a significant challenge \citep{castro2022lifting}. In what follows, inspired by the approach in \cite{castro2022lifting}, we explain how to leverage $\D_{\gph \phi}$ to compute tilting quantities.
Assume we have a network $\D_{\gph \phi} = (\N, \A, l, \nu)$ and a vector $(\alpha, \gamma, \beta)$ defining an inequality of the form \eqref{eq:general-cut}. Let $\ell(\alpha, \gamma): \A \rightarrow \R^{|\A|}$ be a function that maps each arc $a \in \A$ to a length in the following way:
 $$
    \ell_a(\alpha, \gamma) = \begin{cases}
        \nu_a\alpha_{\idx(a)} & \text{if } a \in \A_L, \\
        \nu_a \gamma & \text{if } a \in \A_F.
    \end{cases}
$$

Let $\D_{\gph \phi}(\alpha, \gamma) = (\N, \A, \ell(\alpha, \gamma), \nu)$ be an arc length-parametrized version of $\D_{\gph \phi}$, \ie, it inherits the node and arc sets, as well as the arc values, but redefines the arc lengths. Observe that each $r$-$t$ path in $\D_{\gph \phi}(\alpha, \gamma)$ still uniquely maps to a vector $(x, z) \in \gph \phi$, and its length equals $\alpha^\top x + \gamma z$. Now, for any node $u \in \N \setminus \{ r, t \}$, let $\sigma_u^{\downarrow}$ denote the longest path length from the root node $r$ to $u$, and $\sigma_u^{\uparrow}$ the longest path length from $u$ to the sink node $t$, both in $\D_{\gph \phi}(\alpha, \gamma)$. Moreover, for any arc $a \in \A_L$, let
$$
    \mu_a := \max_{w \geq 0} \left\{ \sum_{a' \in \A} \ell_{a'}(\alpha, \gamma)\, w_{a'} : w \text{ represents an $r$-$t$ path in $\D_{\gph \phi}(\alpha, \gamma)$},\ w_a = 1 \right\}
$$
be the maximum length of any $r$-$t$ path in $\D_{\gph \phi}(\alpha, \gamma)$ traversing arc $a$. For any $j \in [n_L]$ and $\theta \in \bin$, we have
\begin{center}$
    t^\theta_j(\alpha, \gamma, \beta) = \beta - \max\{\mu_a : a \in \A_L,\ \idx(a) = j,\ \nu_a = \theta\}.$
\end{center} 
\medskip

By definition, we have $\mu_{uv} = \sigma_u^{\downarrow} + \sigma_v^{\uparrow} + \ell_{uv}(\alpha, \gamma)$ for any arc $(u, v) \in \A$. Hence, we obtain 
\medskip
\begin{center}$
\begin{array}{ll}
    t^\theta_j(\alpha, \gamma, \beta) & = \beta - \max\{\sigma_u^{\downarrow} + \sigma_v^{\uparrow} + \ell_{uv}(\alpha, \gamma): (u, v) \in \A_L,\ \idx(u, v) = j,\ \nu_{uv} = \theta\} \\
    & = \beta - \theta \alpha_j - \max\{\sigma_u^{\downarrow} + \sigma_v^{\uparrow}: (u, v) \in \A_L,\ \idx(u, v) = j,\ \nu_{uv} = \theta\}
\end{array}$
\end{center}
\medskip
for $\theta \in \bin$. Observe that computing $\sigma_u^{\downarrow}$ or $\sigma_v^{\uparrow}$ for a fixed vector $(\alpha, \gamma)$ corresponds to solving a longest path problem over $\D_{\gph \phi}$. Therefore, the time complexity of computing $\mu_a$ for any arc $a \in \A_L$ is linear in the size of $\D_{\gph \phi}$. This, in turn, implies that for fixed $n_L$, the time complexity of computing all tilting quantity functions is also linear in the size of $\D_{\gph \phi}$.

\begin{example}
    We continue the running example by indicating how to compute tilting quantity functions using the network in Figure \ref{fig:exampleNetwork}. Consider an initial valid inequality $-x + 3y \leq 6$, which is tight at $x = 3$ ($(x_1, x_2) = (1, 1)$ for the binary expansion). This inequality is depicted as a black line in Figure \ref{fig:example1-cutTilting}. Vectors $\sigma^{\downarrow}$ and $\sigma^{\uparrow}$ are computed by solving longest path problems over $\D_{\gph \phi}(\alpha, \gamma)$, where $\alpha = (-1, -2)$ and $\gamma = 3$ are illustrated as black arc labels in Figure \ref{fig:example1-networkTilting}. We have
    \medskip    %
    \begin{center}$
    \begin{array}{lllllll}
        \text{(Layer 1)} & \sigma^{\downarrow}_r = 0, \\
        \text{(Layer 2)} & \sigma^{\downarrow}_1 = 0, & \sigma^{\uparrow}_1 = -2, & \sigma^{\downarrow}_2 = -1, &\sigma^{\uparrow}_2 = 7, \\
        \text{(Layer 3)} & \sigma^{\downarrow}_3 = -2, & \sigma^{\uparrow}_3 = -0, & \sigma^{\downarrow}_4 = 0, & \sigma^{\uparrow}_4 = 3, & \sigma^{\downarrow}_5 = -3, &  \sigma^{\uparrow}_5 = 9, \\
        \text{(Layer 4)} & \sigma^{\uparrow}_t = 0.
    \end{array}$
    \end{center}
    \medskip

    We can now compute all tilting quantity functions:
    \medskip
    \begin{center}$
    \begin{array}{l}
        t^0_1(\alpha, \gamma, \beta) = 6 - \max\{ -2, 3 \} = 3, \\ 
        t^1_1(\alpha, \gamma, \beta) = 6 - (-1) - \max\{ 3, 7 \} = 0, \\
        t^0_2(\alpha, \gamma, \beta) = 6 - \max\{ 3, 2 \} = 3, \\
        t^1_2(\alpha, \gamma, \beta) = 6 - (-2) - \max\{ 0, 8 \} = 0.
    \end{array}$
    \end{center}
\medskip

    As $t_1^1, t_2^1 = 0$ and $t_1^0, t_2^0 = 3$, the inequality can only be tilted with respect to $(1 - x_1)$ or $(1 - x_2)$ by at most 3 units. Tilting with respect to $(1 - x_1)$ by 1 unit yields the tilted inequality induced by $(\widehat{\alpha}, \gamma, \widehat{\beta})$, where $\widehat{\alpha} = (\alpha_1 - 1, \alpha_2)$ and $\widehat{\beta} = \beta - 1$. By repeating this procedure over $\D_{\gph \phi}(\widehat{\alpha}, \gamma)$, we can tilt the resulting inequality with respect to $(1 - x_2)$ by 2 units, achieving the final tilted inequality $-2x_1 - 4x_2 + 3y \leq 3$, illustrated as an orange line in Figure \ref{fig:example1-cutTilting}. \qedwhite
\end{example}

\begin{figure}[t]
    \centering
    \begin{subfigure}[t]{0.5\textwidth}
        \centering
        \begin{tikzpicture}[scale=1.25]
            \draw[very thin, color=gray] (-.9, -.1) grid (3.9, 3.9); 
        
            \draw[->, line width=1] (-1.1,0) -- (4, 0) node[right] {$x$};
            \draw[->, line width=1] (0, 0) -- (0, 4) node[above] {$y$};
            \foreach \x in {0, 1, 2, 3}
                \draw (\x cm, -1pt) -- (\x cm,-3pt) node[anchor=north] {$\x$};
            \foreach \y in {1, 2, 3}
                \node[anchor=east] at (-.9, \y cm) {$\y$};
        
            \draw[fill=myblue, opacity=0.15, dashed] (0, 0) -- (0, 1) -- (3, 3) -- (3, 0) -- (0, 0);
        
            \draw[line width=2] (-1, 1.67) -- (3.9, 3.29);
            \draw[myorange, line width=2] (-1, 0.33) -- (3.9, 3.6);
        
            \draw[->, dashed, line width=2, bend right=40, myorange] (0.5, 2.15) to (0.5, 1.4);
        
            \draw[line width=2.5] (-0.5, 1) -- (0, 3) -- (2, 3.5) -- (3, 3) -- (2.25, 0) -- (1.5, 0) -- (-0.52, 1.02);

            \draw[->, line width=2, myblue] (0, 1) -- (0, 0);
            \draw[->, line width=2, myblue] (1, 1) -- (1, 0);
            \draw[->, line width=2, myblue] (3, 3) -- (3, 0);
            
            \node[fill=black, circle, scale=.6] at (0, 2) {};
            \node[fill=black, circle, scale=.6] at (0, 3) {};
            \node[fill=black, circle, scale=.6] at (1, 2) {};
            \node[fill=black, circle, scale=.6] at (1, 3) {};
            \node[fill=black, circle, scale=.6] at (2, 1) {};
            \node[fill=black, circle, scale=.6] at (2, 2) {};    
            \node[fill=black, circle, scale=.6] at (2, 3) {};
        
            \node[fill=myblue, circle, scale=.6] at (0, 1) {};
            \node[fill=myblue, star, scale=.6] at (1, 1) {};
            \node[fill=myblue, circle, scale=.6] at (2, 0) {};    
            \node[fill=myblue, circle, scale=.6] at (3, 3) {};
        
            
        \end{tikzpicture}
        \caption{Initial (black) and tilted (orange) valid inequalities.}
        \label{fig:example1-cutTilting}
    \end{subfigure}%
    ~\hspace{.5cm}
    \begin{subfigure}[t]{0.5\textwidth}
        \centering
        \begin{tikzpicture}[scale=0.85, ->, >=stealth', shorten >=1pt, auto, node distance=3cm, thick, 
                main node/.style={circle, draw, font=\sffamily\footnotesize\bfseries, minimum size=4mm, scale=.75},
                special node/.style={circle, draw, font=\sffamily\bfseries, minimum size=4mm, scale=.85},
                arc label node/.style={text opacity=1, pos=0.5, anchor=center, minimum size=5mm}
            ]
            \node[special node] (r) at (0, 0) {r};
            
            \node[main node] (1) at (-1.5, -2) {1};
            \node[main node] (2) at (1.5, -2) {2};
            \node[main node] (3) at (-2, -4) {3};
            \node[main node] (4) at (0, -4) {4};
            \node[main node] (5) at (2, -4) {5};
        
            \node[special node] (t) at (0, -6.5) {t};
        
            \node[] (x1) at (3.5, -1) {$\bm{x_1}$};
            \node[] (x2) at (3.5, -3) {$\bm{x_2}$};
            \node[] (z) at (3.5, -5) {\color{myblue} $\bm{z}$};
        
            \path[every node/.style={font=\sffamily\small,
                fill=white, inner sep=1pt}]
                (r) edge [dashed, line width=1.5] (1)
                    edge [line width=1.5] node [arc label node] {$\bm{-1}$ \color{myorange} $\bm{(-2)}$} (2)
                (1) edge [line width=1.5] node [arc label node] {$\bm{-2}$ \color{myorange} $\bm{(-4)}$} (3)
                    edge [dashed, line width=1.5] (4)
                (2) edge [dashed, line width=1.5] (4)
                    edge [line width=1.5] node [arc label node] {$\bm{-2}$ \color{myorange} $\bm{(-4)}$} (5)
                (3) edge [myblue, line width=1.5] node [arc label node] {$\bm{0}$} (t)
                (4) edge [myblue, line width=1.5] node [arc label node] {$\bm{3}$} (t)
                (5) edge [myblue, line width=1.5] node [arc label node] {$\bm{9}$} (t);
        \end{tikzpicture}
        \caption{Initial (black) and tilted (orange) cut coefficients}
        \label{fig:example1-networkTilting}
    \end{subfigure}
    \vspace{.2cm}
    \caption{Example of tilting valid inequality $-x + 3y \leq 6$ into $-2x + 3y \leq 3$}
    \label{fig:example1-tilting}
\end{figure}
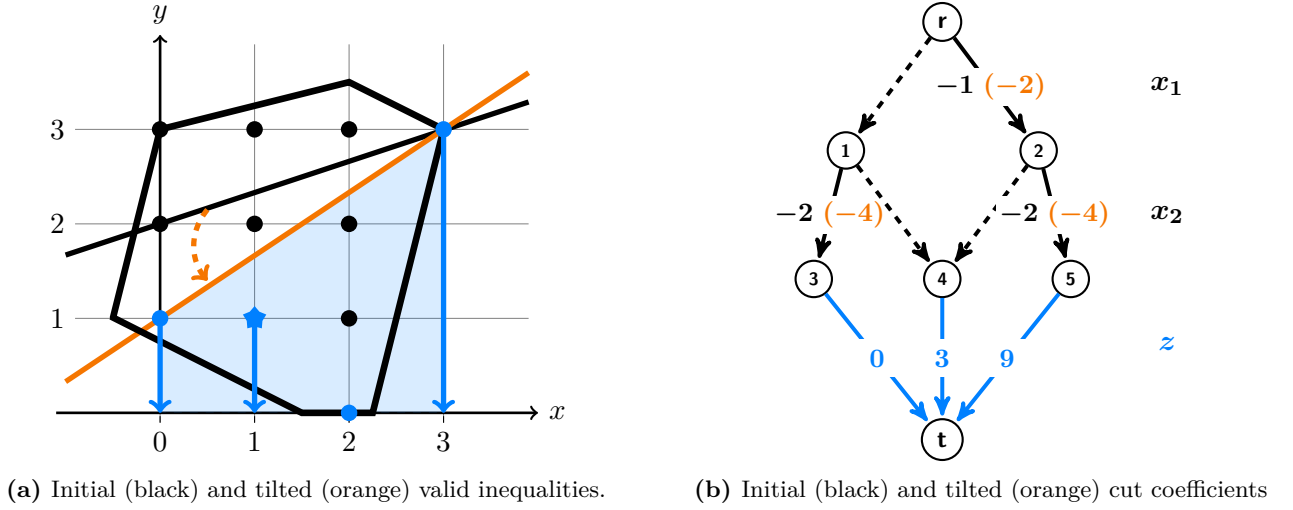

\section{Experimental Results} \label{sec:experiments}
We next present a set of three computational experiments to evaluate the efficiency and strength of the proposed cuts. In the first one, we tackle instances of a multidimensional knapsack bilevel problem, aiming to assess the impact of network size on the available methods by increasing the number of follower constraints. In the second one, we address instances of a binary bilevel program in which the follower's constraint matrix exhibits a block-diagonal structure. As each block can be used to compile a separate network, we assess the impact of the number of networks on the available methods. Finally, we test the effectiveness of including the proposed cuts into state-of-the-art solvers when solving a subset of unsolved problems from the BOBILib benchmark instance library \citep{thurauf2026bobilib}.

We refer to the state-of-the-art solvers of \cite{fischetti2017flms} and \cite{tahernejad2020mibs} as \textsf{FLMS} and \textsf{MibS}, respectively. Both solvers are compiled using their latest version (\textsf{MibS}'s is 1.2.1 and \textsf{FLMS} has a unique version) and are executed using their default settings. We compile networks using Algorithm 2 in \cite{LBC24}, and sequentially strengthen valid inequalities by choosing the coefficient with the smallest tilting quantity, as proposed in \cite{castro2022lifting}. For the first two sets of experiments, we compile exact networks by setting infinite budget. For the last one, we compile relaxed networks (approximate value networks in \cite{LBC24}) according to the procedure proposed in Section 5.2 of \cite{LBC24}. We only include results for our cutting plane approach using tilted cuts, as they significantly outperform the non-tilted version.

Our algorithms are implemented in C\texttt{++} and all computations are executed on a 12-core Intel(R) Xeon(R) Gold 6248R CPU \@ 3.00GHz running Linux Ubuntu. All single-level optimization programs required by our algorithms are solved using Gurobi 10, while \textsf{FLMS} and \textsf{MibS} use CPLEX 12.

\subsection{Multidimensional knapsack bilevel problem}
We first study a class of multidimensional knapsack bilevel problems in which the leader problem has no constraints, \ie, $m_L = 0$, and the follower problem has a number of knapsack-type constraints $m_F \in \{1, 2, 3, 4\}$. The purpose of this experiment is to evaluate and compare the performance of available network flow-based approaches for different network sizes. Due to the nature of the network construction algorithm in \cite{LBC24}, a higher value of $m_F$ leads to larger networks.

\medskip

\noindent \textbf{Instances.} 
We consider $n_L = n_F = 30$ and denote as $\mathrm{DU}[l, u]$ the discrete uniform distribution on the integers $\{ l, \ldots, u \}$. Objective function and follower constraint matrices entries are sampled as follows:
\[
    c_L \sim \mathrm{DU}[-10, 10],\quad c_F \sim \mathrm{DU}[-100, 100],\quad d \sim \mathrm{DU}[-10, -1],
\]
\[
    C \sim \mathrm{DU}[-2, 2],\quad D \sim \mathrm{DU}[0, 10].
\]

The right-hand side vector of the follower's problem is computed as $b = \alpha \cdot (C + D) \mathbf{1}$, where $\alpha \in \{0.25, 0.5, 0.75\}$ is a scaling factor. These values are chosen to create instances in which the leader's influence over the follower's feasible region is not too strong, and the size of the network is well controlled by the tighter bounds on the entries of $C$. We generate 10 instances for each value of $(m_F, \alpha)$, for a total of 120 instances.

\medskip

\noindent  \textbf{Comparison.}
We compare directly solving program \eqref{model:LBC-graphReformulation} using an off-the-shelf MIP solver, and a cutting-plane approach using the proposed cut families introduced in Section \ref{sec:new-reformulations} after applying the tilting procedure described in Section \ref{sec:tiltingFunctionsWithDDs}. We differentiate between ``no-good cuts'' (NGC) and enhanced optimality cuts; see Section \ref{sec:optCuts}.
We also compare target cuts \citep{buchheim2008localcuts,tjandraatmadja2019target}, which can be separated by solving the dual of a network flow program over the network. If $\conv(\dom \Y)$ is full-dimensional, then target cuts are guaranteed to be facet-defining \citep{tjandraatmadja2019target}. 

Observe that, as we compile exact networks, all these solution methods are exact. However, in general, 
we may not be capable of solving program \eqref{model:LBC-graphReformulation} within the time limit, or we may use a relaxed network. In both cases, the result yields a valid lower bound. Hence, we are interested in evaluating what portion of this bound is recovered within the time limit. Let $\D$ be the compiled network, and $v^\ast(\D)$ the optimal value of program \eqref{model:LBC-graphReformulation} using $\D$ instead of $\D_{\gph \phi}$. Moreover, consider a time limit $T > 0$ and a set $S$ of solution methods. We define the bound recovery ratio for method $s \in S$ at time $t > 0$ as
$$
    r^s_t(\D) = \frac{v^s_{\texttt{LB}, t}(\D) - \text{LB}}{\text{UB} - \text{LB}},
$$
where $v^s_{\texttt{LB}, t}(\D)$ and $v^s_{\texttt{UB}, t}(\D)$ are the lower and upper bounds achieved by method $s$ at time $t$, respectively,~$\text{LB} = \min_{s \in S}\{v_{\texttt{LB}, 0}^s(\D)\}$ is the lowest initial lower bound achieved by any method (\eg, the LP relaxation), and~$\text{UB} = \min_{s \in S} \{v^s_{\texttt{UB}, T}(\D) \}$ is the best value (upper bound) achieved by any solution method. We use the convention $r_t^s(\D) = 1$ when $\text{UB} = \text{LB}$, in which case optimality is already certified. This way, we have $r_t^s(\D) \in [0, 1]$ and $v^\ast(\D) \in [v^s_{\texttt{LB}, t}(\D), \text{UB}]$ for any $t \in [0, T]$ and $s \in S$. Moreover, $r_t^s(\D) = 1$ if and only if $v_{\texttt{LB}, t}^s(\D) = \text{UB} = v^\ast(\D)$. Thus, $r_t^s(\D) = 1$ indicates that method $s$ achieves the optimal value of program \eqref{model:LBC-graphReformulation} using $\D$ by time $t$.

\medskip

\noindent \textbf{Results.} 
Figure \ref{fig:results-multirow} illustrates average values of the bound recovery measure $r$ for different $m_F$ values with a time limit $T = 3,600$. This time limit does not consider the time used to compile the network. We observe that, for small size networks ($m_F \in \{1, 2\}$), directly solving the reformulation \eqref{model:LBC-graphReformulation} corresponds to the fastest approach. For larger networks ($m_F \in \{3, 4\}$), the cutting-plane approaches show a better performance than the reformulation, especially those that do not require solving a linear program. Furthermore, among the proposed cuts, feasibility cuts exhibit an overall faster bound recovery, indicating that the higher cut generation time used to solve the Lagrangian relaxation gets paid off by providing a stronger seed inequality.

\begin{figure*}[t]
    \centering
    \begin{subfigure}[t]{0.48\textwidth}
        \centering
        \includegraphics[width=\textwidth]{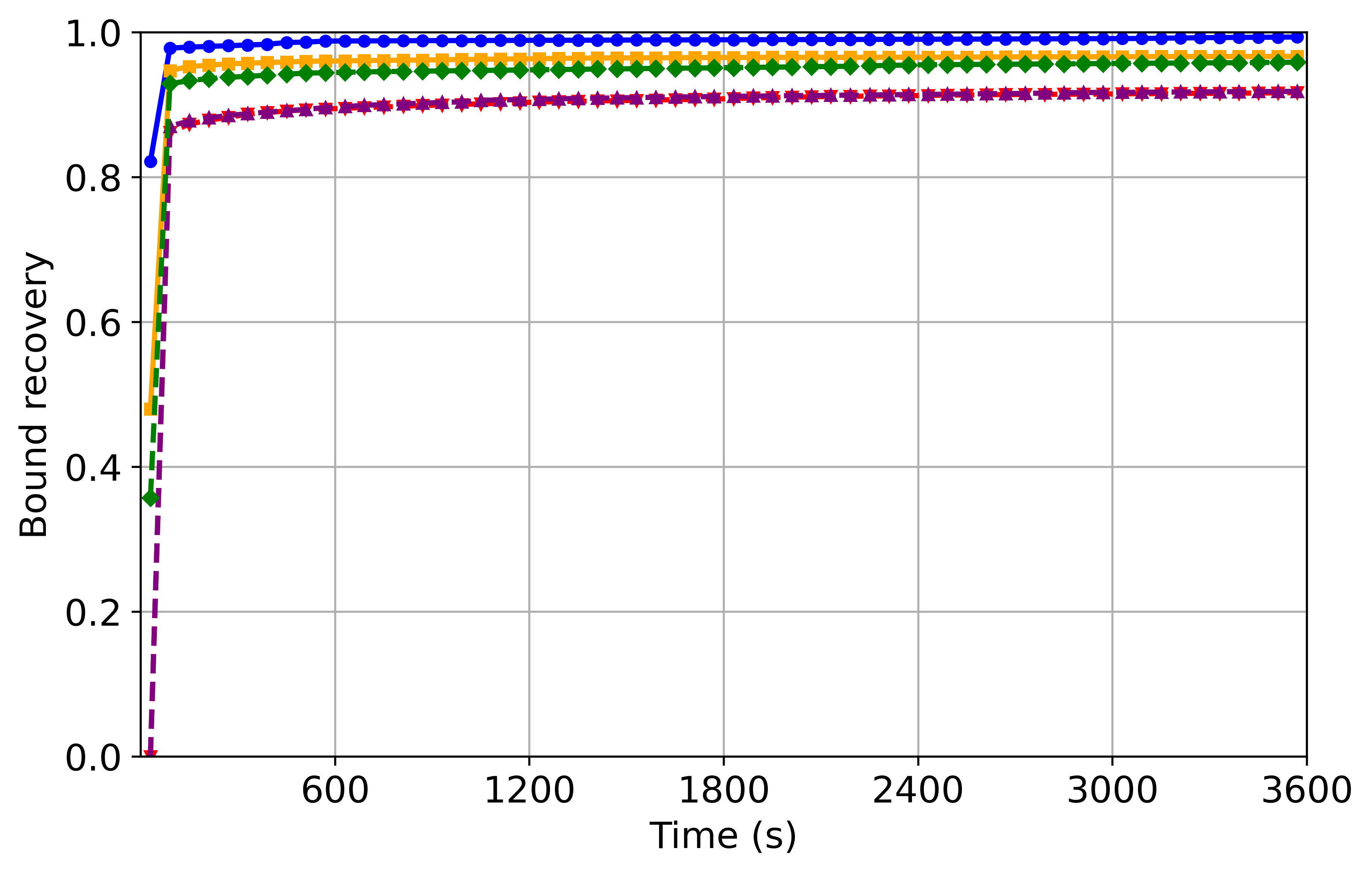}
        \caption{$m_F = 1$}
    \end{subfigure}%
    ~ 
    \begin{subfigure}[t]{0.48\textwidth}
        \centering
        \includegraphics[width=\textwidth]{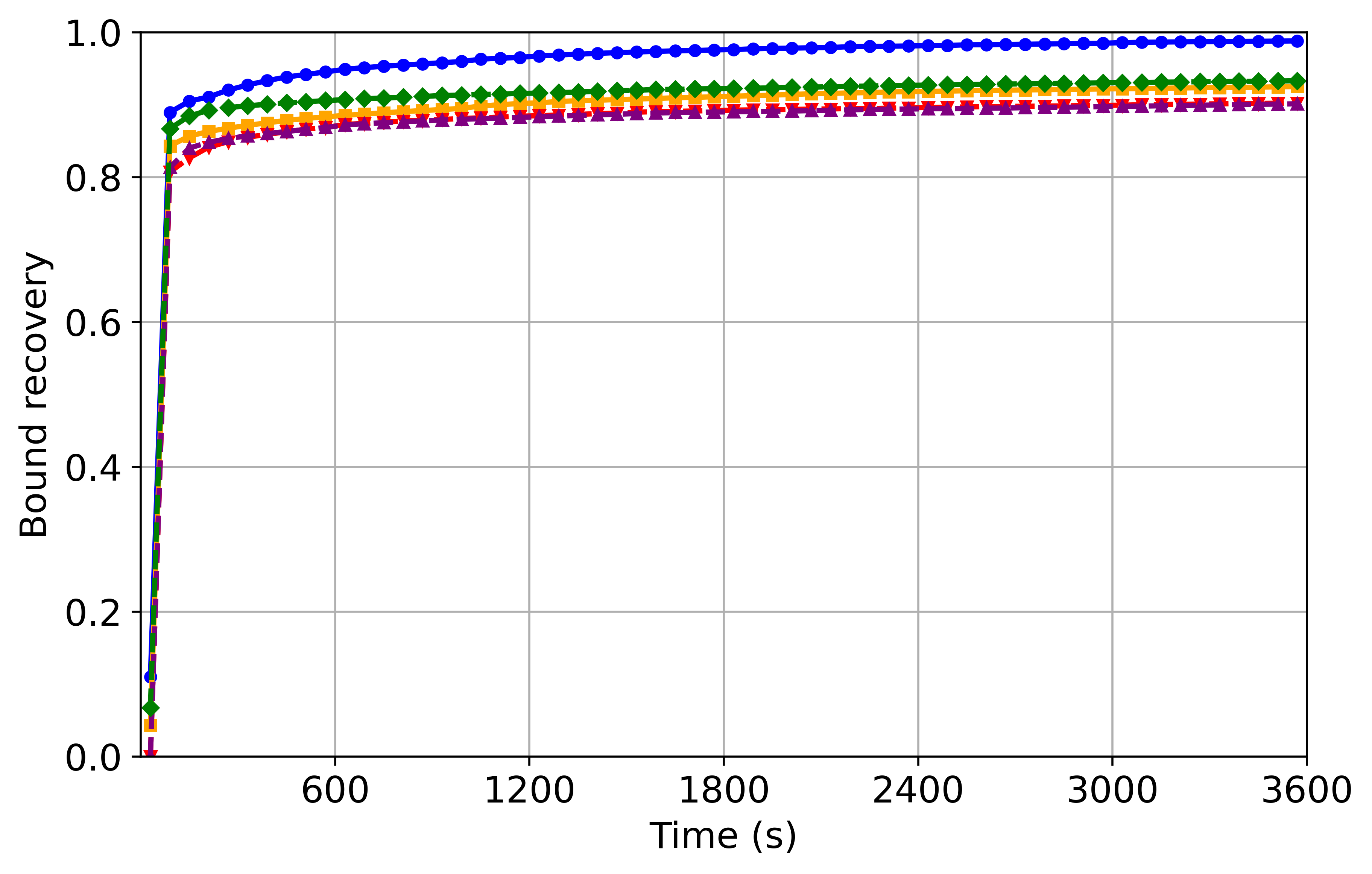}
        \caption{$m_F = 2$}
    \end{subfigure}
    ~
    \begin{subfigure}[t]{0.48\textwidth}
        \centering
        \includegraphics[width=\textwidth]{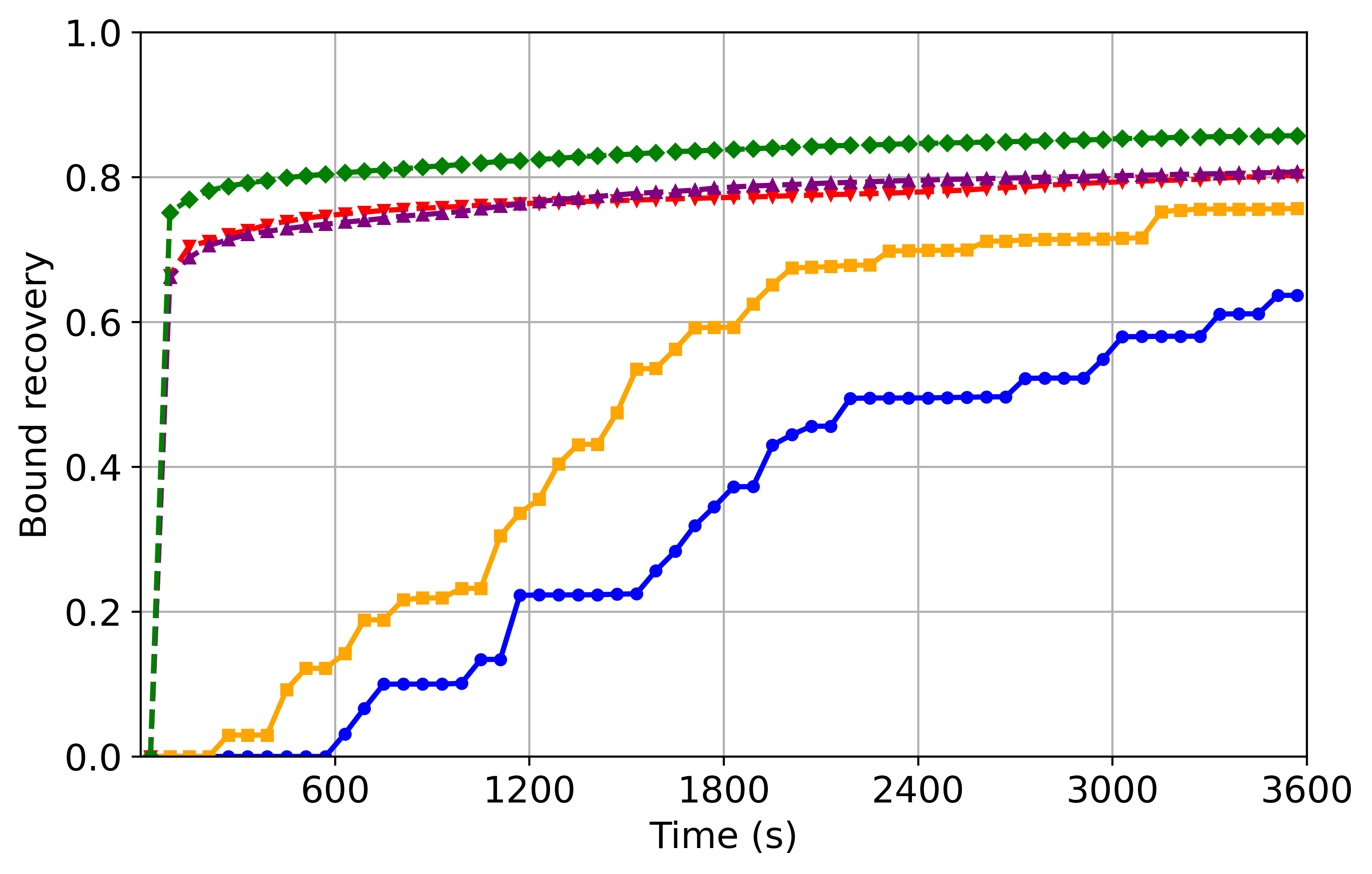}
        \caption{$m_F = 3$}
    \end{subfigure}
    ~
    \begin{subfigure}[t]{0.48\textwidth}
        \centering
        \includegraphics[width=\textwidth]{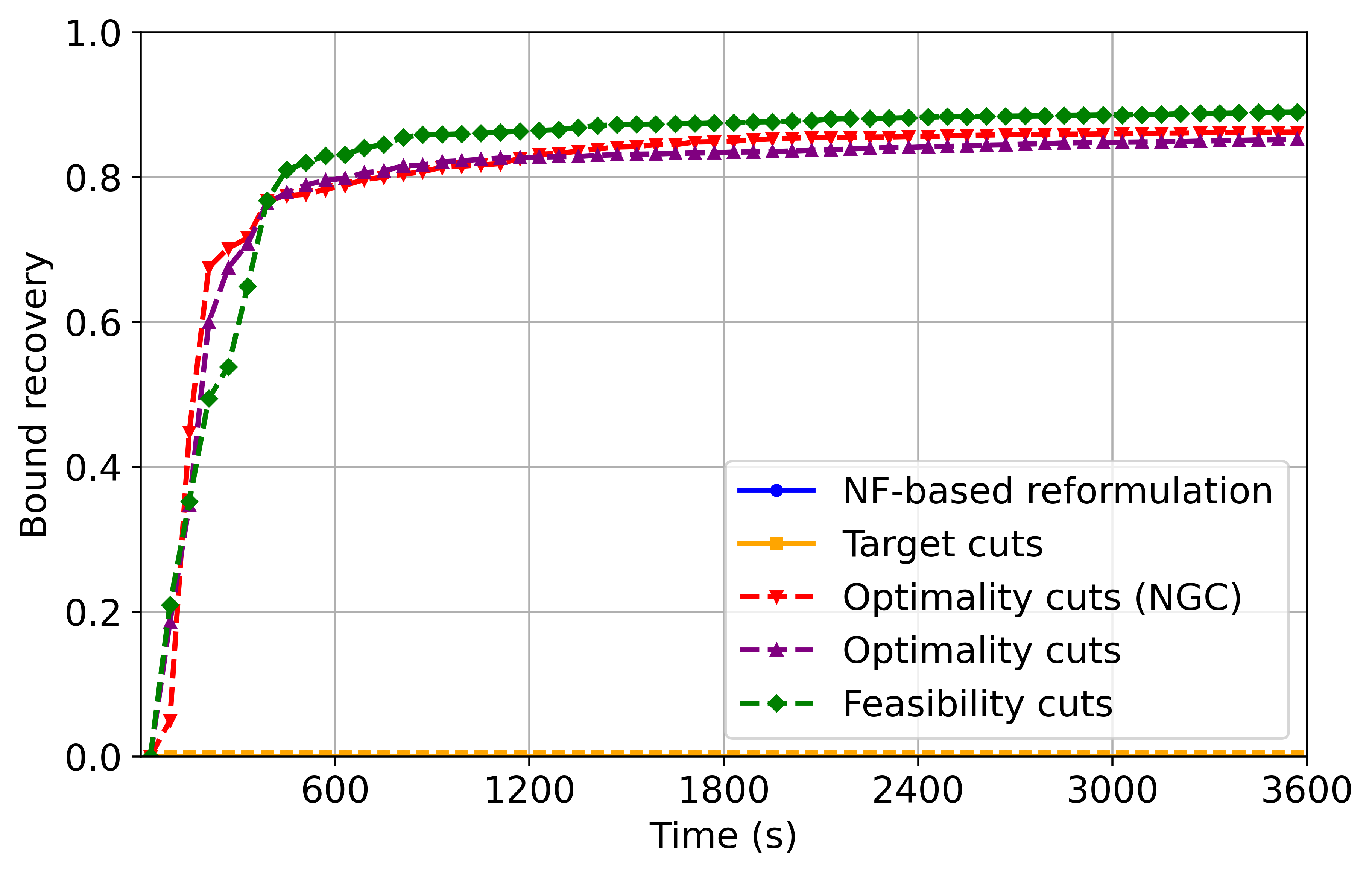}
        \caption{$m_F = 4$}
    \end{subfigure}
    \vspace{.2cm}
    \caption{Results for the multidimensional knapsack instances}
    \label{fig:results-multirow}
\end{figure*}

\subsection{Block-diagonal follower bilevel problem}
We now study a class of binary bilevel problems in which the follower's problem has a block-diagonal constraint matrix. We can alternatively interpret these problems as binary bilevel programs with multiple independent followers~\citep{calvete2007multifollower}. The objective of this experiment is to assess the effect of decomposable structures in the follower's problem. Let $k > 0$ be the number of follower blocks. In these problems, the follower's problem constraint matrix $D$ has a block-diagonal structure with $k$ blocks, each with a knapsack constraint, \ie, $m_F = k$. Accordingly, the follower's objective function vector $d$ and right-hand side vector $b$, as well as the leader's objective function vector $c_F$ and constraint matrix $B$, are split into $k$ smaller vectors (matrices) $\{ d^p \}_{p \in [k]}$, $\{b^p\}_{p \in [k]}$, $\{ c_F^p \}_{p \in [k]}$, and $\{B^p\}_{p \in [k]}$, respectively, of appropriate dimension. Matrix $C$ is defined for each block as $\{C^p\}_{p \in [k]}$. This way, we can rewrite the follower's value function as $\phi(x) = \sum_{p \in [k]} \phi^p(x)$, where
$$
    \phi^p(x) = \min\{ d^{p^\top} y^p : D^p y^p \leq b^p - C^p x \}
$$
is the optimal value of the follower's problem associated with the $p$-th block. By extending Theorem 3.1 in \cite{calvete2007multifollower} to the binary setting, we obtain that~$(x, z) \in \hyp \phi$ if and only if there exist~$z^p$ such that~$(x, z^p) \in \hyp \phi^p$ for every $p \in [k]$ and~$z = \sum_{p \in [k]} z^p$. Hence, we can equivalently write program \eqref{model:hypConvReformulation} as
\begin{alignat*}{3}
    \min & & \hspace{.2cm} & c_L^\top x + \sum_{p \in [k]} c_F^{p^\top} y^p \\
    \st & & & Ax + \sum_{p \in [k]} B^p y^p \leq a \\
    & & & C^p x + D^p y^p \leq b^p & \hspace{.5cm} & p \in [k] \\
    & & & (x, d^{p^\top} y^p) \in \conv( \hyp \phi^p ) && p \in [k].
\end{alignat*}

As it is done for $(x, d^\top y) \in \conv(\hyp \phi)$, we can enforce each constraint $(x, d^{p^\top} y^p) \in \conv(\hyp \phi^p)$ with a network $\D_{\gph \phi^p} = (\N^p, \A^p, l^p, \nu^p)$ compiled only considering the $p$-th block problem parameters. Hence, we can rewrite the network flow-based reformulation \eqref{model:LBC-graphReformulation}, for this setting, as
\begin{alignat*}{3}
    \min & & \hspace{.2cm} & c_L^\top x + \sum_{p \in [k]} c_F^{p^\top} y^p \\
    \st & & & Ax + \sum_{p \in [k]} B^p y^p \leq a \\
    & & & C^p x + D^p y^p \leq b^p & \hspace{.5cm} & p \in [k] \\
    & & & \sum_{a \in \delta^+(u)} w^p_a - \sum_{a \in \delta^-(u)} w^p_a = \mathbf{1}(u = r^p) - \mathbf{1}(u = t^p) & \hspace{.5cm} & u \in \N^p \\
    & & & x_j = \sum_{a \in \A^p_L: \idx(a) = j} \nu^p_a w^p_a & & j \in [n_L],\ p \in [k] \\
    & & & d^{p^\top} y^p \leq l^{p^\top} w^p && p \in [k] \\
    & & & w^p \geq 0 && p \in [k].
\end{alignat*}

Similarly, we can separate optimality and feasibility cuts using each network $\D_{\gph \phi^p}$, as Propositions \ref{thm:optimalityReformulationEquivalence} and \ref{thm:feasibilityReformulationEquivalence} hold.

\medskip

\noindent \textbf{Instances.}
We consider $n_L = n_F \in \{30, 45, 60, 75\}$. Objective function and follower constraint matrices entries are sampled as follows:
\[
    c_L, c_F \sim \mathrm{DU}[-100, 100],\quad d \sim \mathrm{DU}[-100, -1],\quad C \sim \mathrm{DU}[0, 10],\quad D \sim \mathrm{DU}[0, 100].
\]

As in the previous problem class, the right-hand side vector of the follower's problem is computed as $b = \alpha \cdot (C + D) \mathbf{1}$, where $\alpha \in \{0.25, 0.5, 0.75\}$ is a scaling factor. These values are chosen to achieve a relatively generic set of instances, where the leader problem has no constraints. We generate 10 instances for each value of $(n_F, k, \alpha)$, for a total of 360 instances.

\medskip

\noindent \textbf{Comparison.}
We compare directly solving program \eqref{model:LBC-graphReformulation}, the cutting plane approach using feasibility cuts, and state-of-the-art bilevel solvers \textsf{FLMS} \citep{fischetti2017flms} and \textsf{MibS} \citep{tahernejad2020mibs}. Furthermore, we include two extra variants of each bilevel solver: (i) we replace the HPR with reformulation \eqref{model:LBC-graphReformulation}, and (ii) we add feasibility cuts to the HPR. Networks compilation and cuts generation (if applicable) is done within a time limit of 300 seconds, after which the corresponding bilevel solver is executed for 3,300 seconds. Both variants use the same networks. The variant of \textsf{MibS} with reformulation \eqref{model:LBC-graphReformulation} is excluded as it fails to produce finite gaps for several instances within the time limit.

\medskip

\noindent \textbf{Results.}
Figures \ref{fig:results-multirow_part1} and \ref{fig:results-multirow_part2} present two-panel performance plots. The left panes show the cumulative number of solved instances, and the right panes the optimality gap of the unsolved ones. We limit the maximum gap to 100\% for illustrative purposes. Overall, when $k = 1$, directly solving reformulation \eqref{model:LBC-graphReformulation} yields the best results, especially for $n = 75$. For any other configuration, the cutting plane approach outperforms the other methods. 
In particular, the results for $k > 1$ align with the fact that the cutting plane approach fully decomposes the convexification of $\hyp \phi$ and the information retrieval process into different networks, whereas the reformulation \eqref{model:LBC-graphReformulation} manages all networks within the same program, yielding, on average, a higher number of variables as $k$ increases.
Regarding state-of-the-art solvers, we observe that including feasibility cuts into the initial reformulation significantly improves its performance, whereas replacing the HPR with \eqref{model:LBC-graphReformulation}, in general, hurts it. Even though reformulation \eqref{model:LBC-graphReformulation} provides an exact representation of $\conv(\hyp \phi)$, the extended variable space and the larger formulation greatly impact the effectiveness of these solvers. This also highlights the benefits of retrieving the networks' information through cuts that are projected into the original variable space.

\begin{figure*}[ht!]
    \centering
    \begin{subfigure}[t]{0.5\textwidth}
        \centering
        \includegraphics[width=\textwidth]{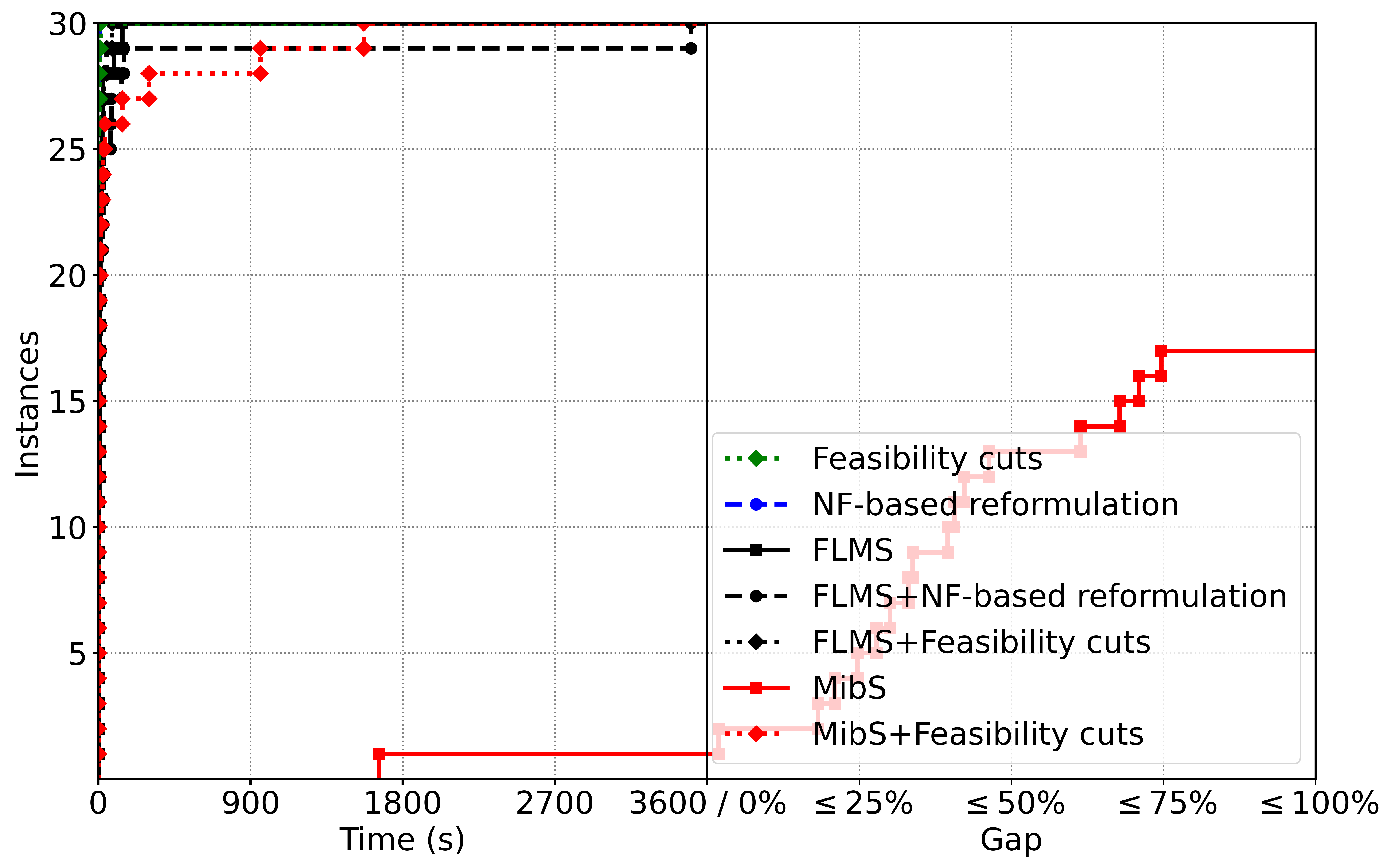}
        \caption{$n = 30,\ k  = 1$}
    \end{subfigure}%
    ~ 
    \begin{subfigure}[t]{0.5\textwidth}
        \centering
        \includegraphics[width=\textwidth]{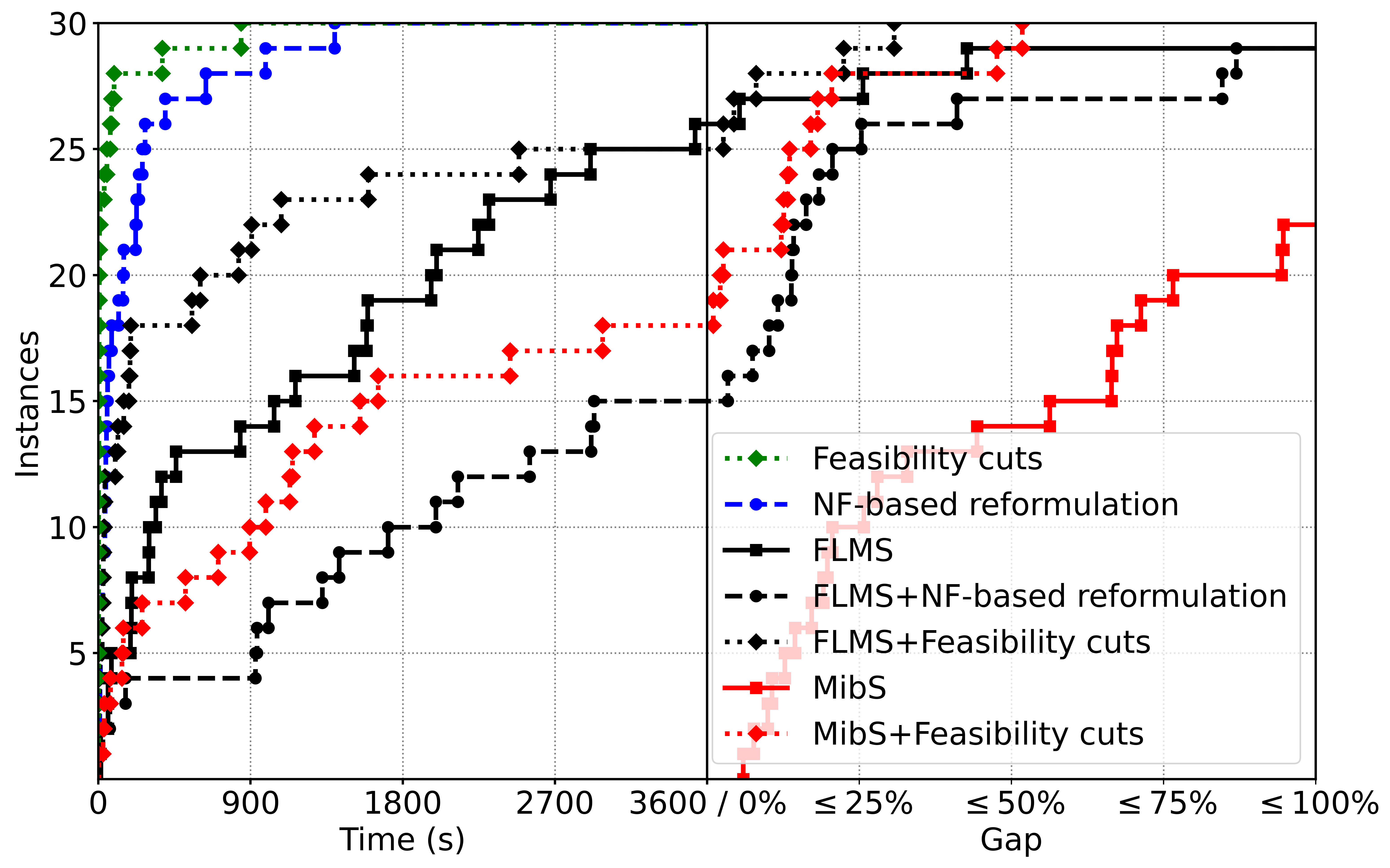}
        \caption{$n = 30,\ k = 3$}
    \end{subfigure}%
    \hfill \\
    \begin{subfigure}[t]{0.5\textwidth}
        \centering
        \includegraphics[width=\textwidth]{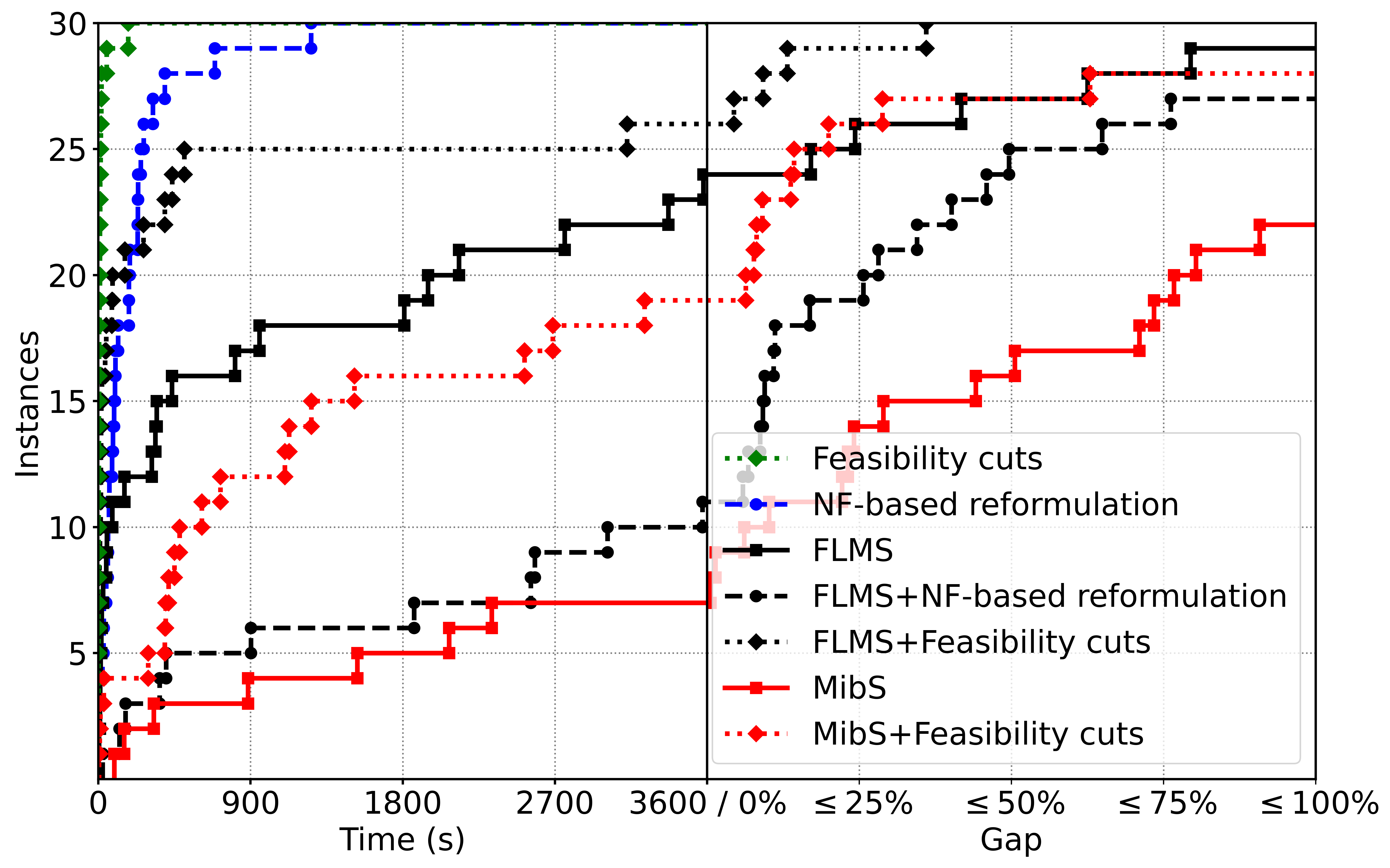}
        \caption{$n = 30,\ k = 5$}
    \end{subfigure}%
    ~
    \begin{subfigure}[t]{0.5\textwidth}
        \centering
        \includegraphics[width=\textwidth]{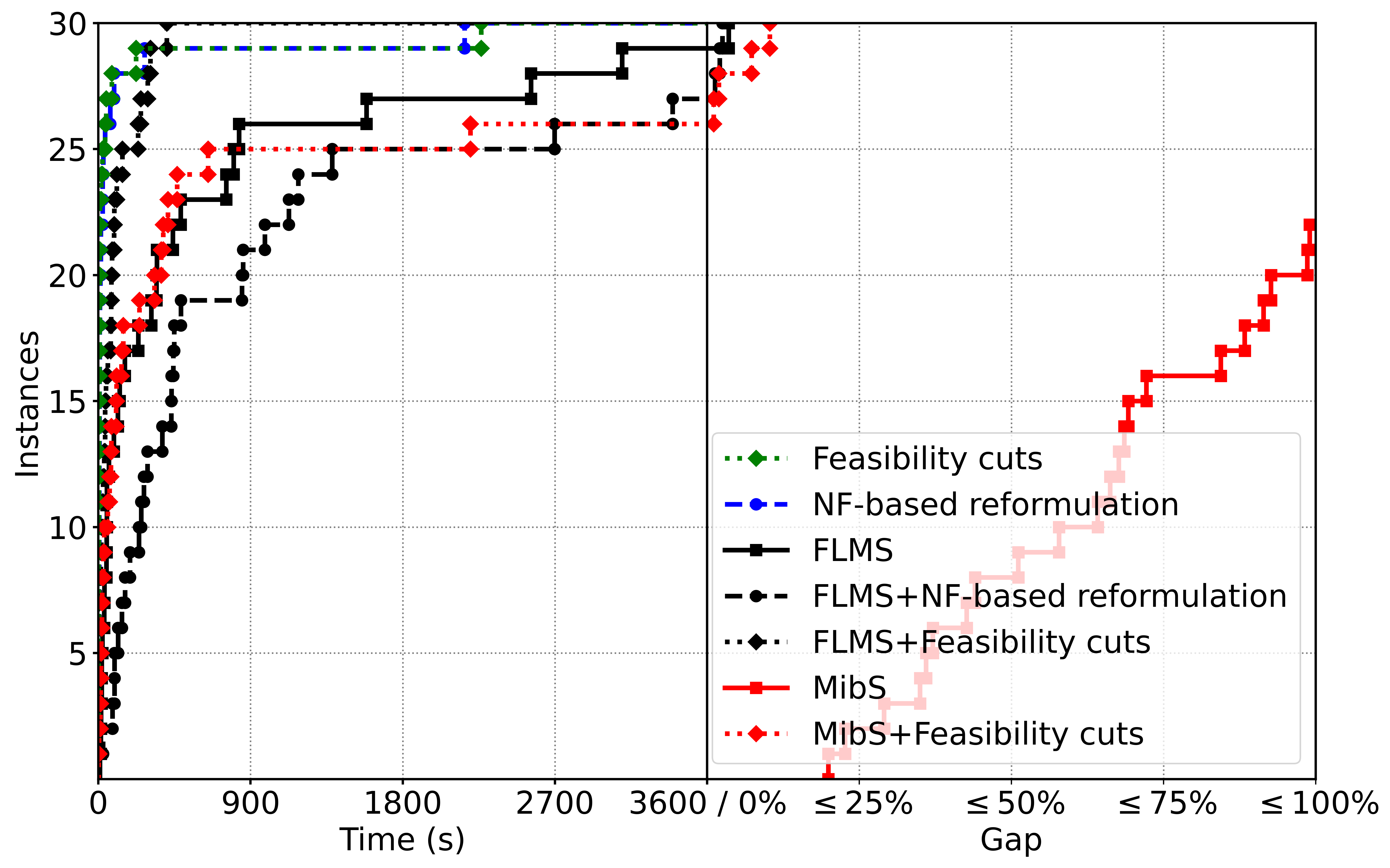}
        \caption{$n = 45,\ k = 1$}
    \end{subfigure}%
    \hfill \\
    \begin{subfigure}[t]{0.5\textwidth}
        \centering
        \includegraphics[width=\textwidth]{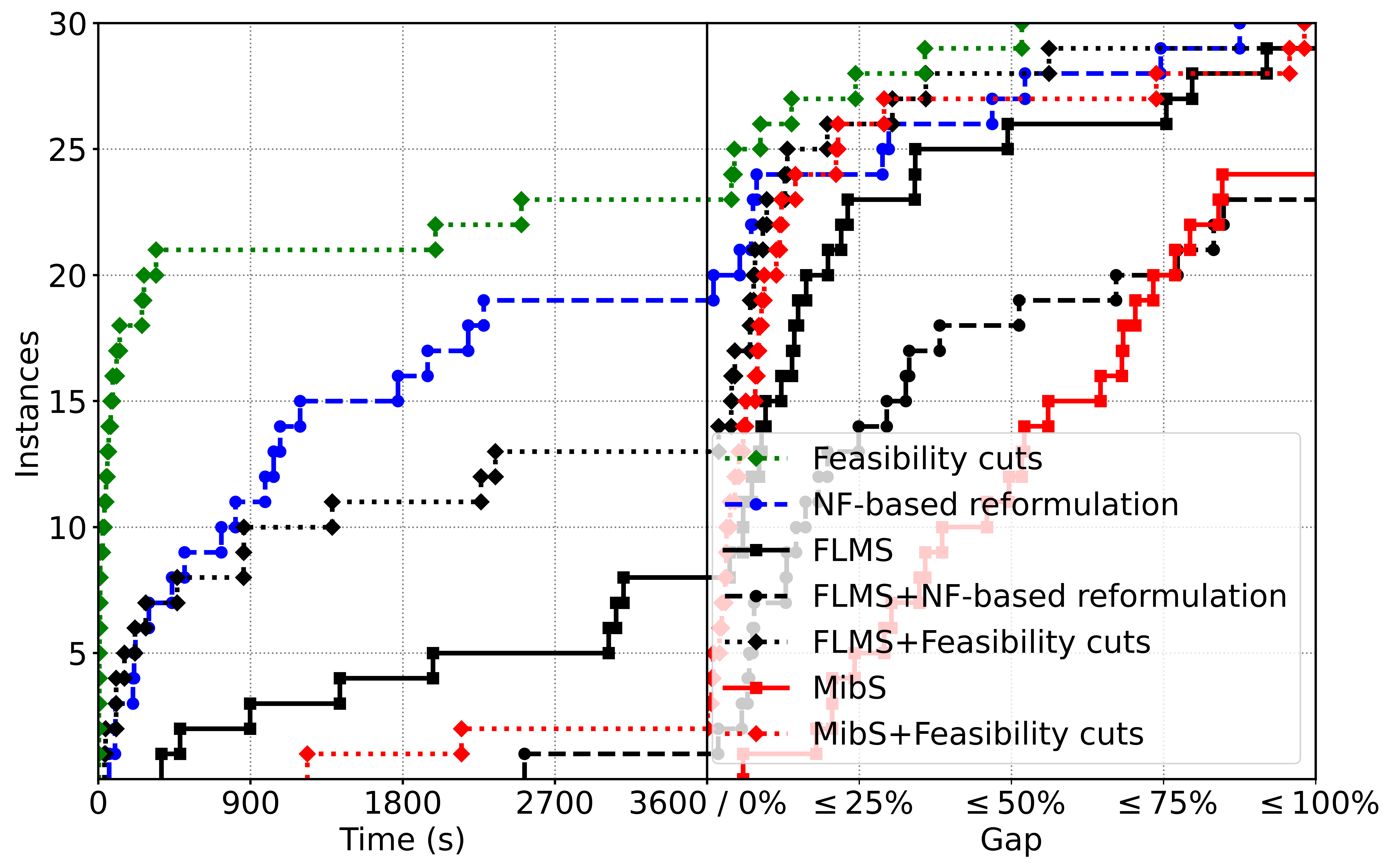}
        \caption{$n = 45,\ k = 3$}
    \end{subfigure}%
    ~
    \begin{subfigure}[t]{0.5\textwidth}
        \centering
        \includegraphics[width=\textwidth]{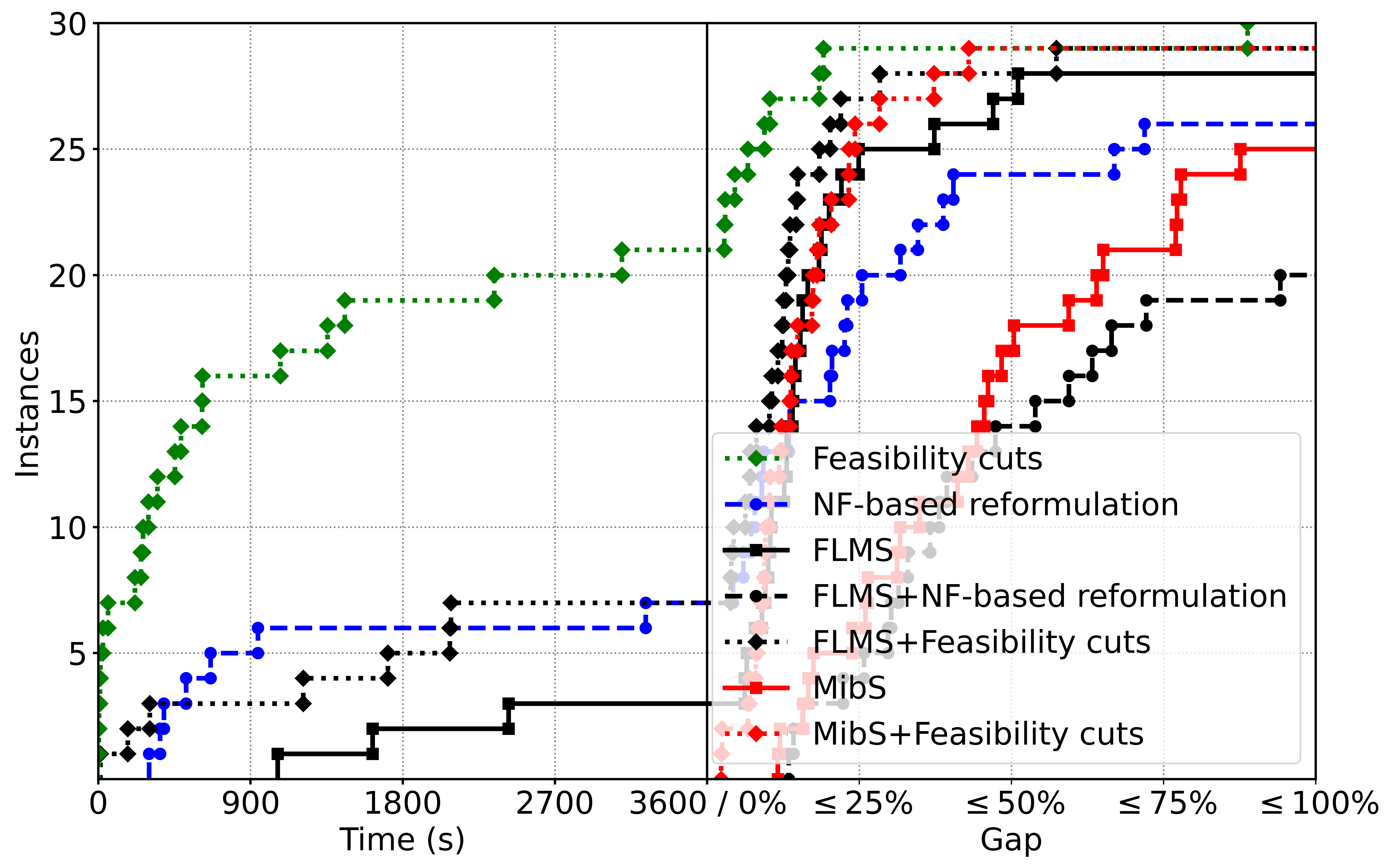}
        \caption{$n = 45,\ k = 5$}
    \end{subfigure}%
    \vspace{.2cm}
    \caption{Results for the block-diagonal follower instances}
    \label{fig:results-multirow_part1}
\end{figure*}

\begin{figure*}[ht!]
    \centering
    \begin{subfigure}[t]{0.5\textwidth}
        \centering
        \includegraphics[width=\textwidth]{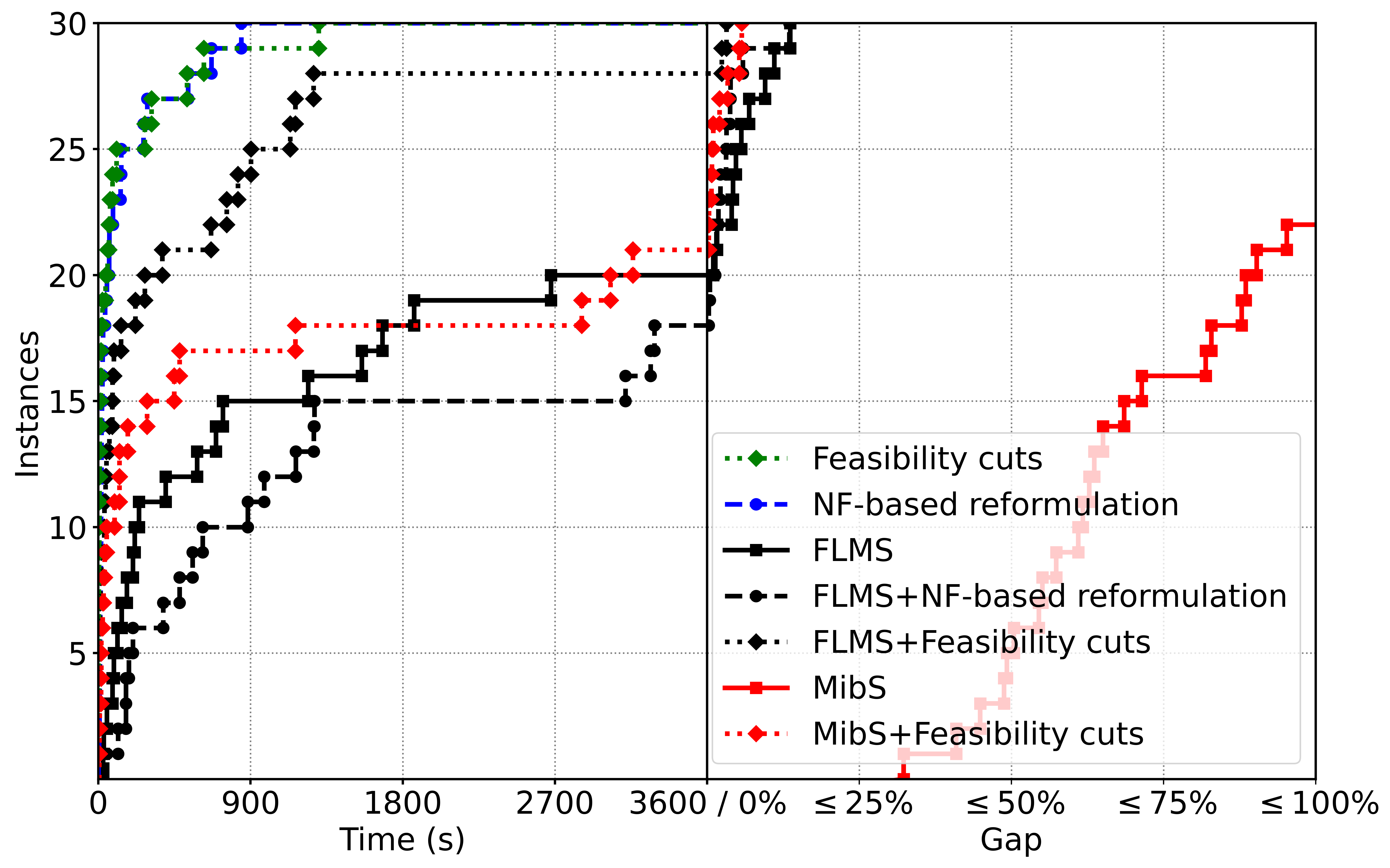}
        \caption{$n = 60,\ k = 1$}
    \end{subfigure}%
    ~
    \begin{subfigure}[t]{0.5\textwidth}
        \centering
        \includegraphics[width=\textwidth]{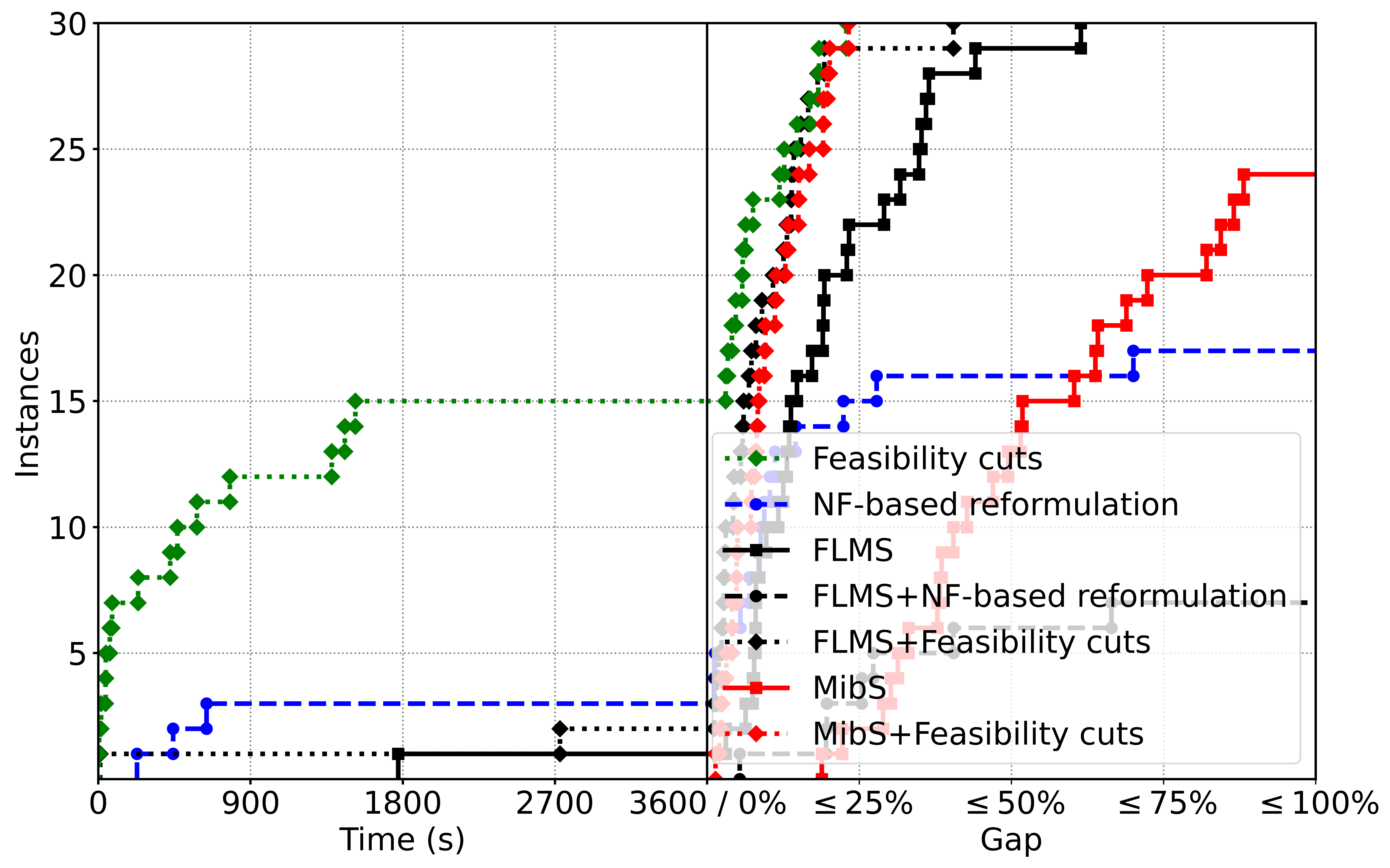}
        \caption{$n = 60,\ k = 3$}
    \end{subfigure}%
    \hfill \\
    \begin{subfigure}[t]{0.5\textwidth}
        \centering
        \includegraphics[width=\textwidth]{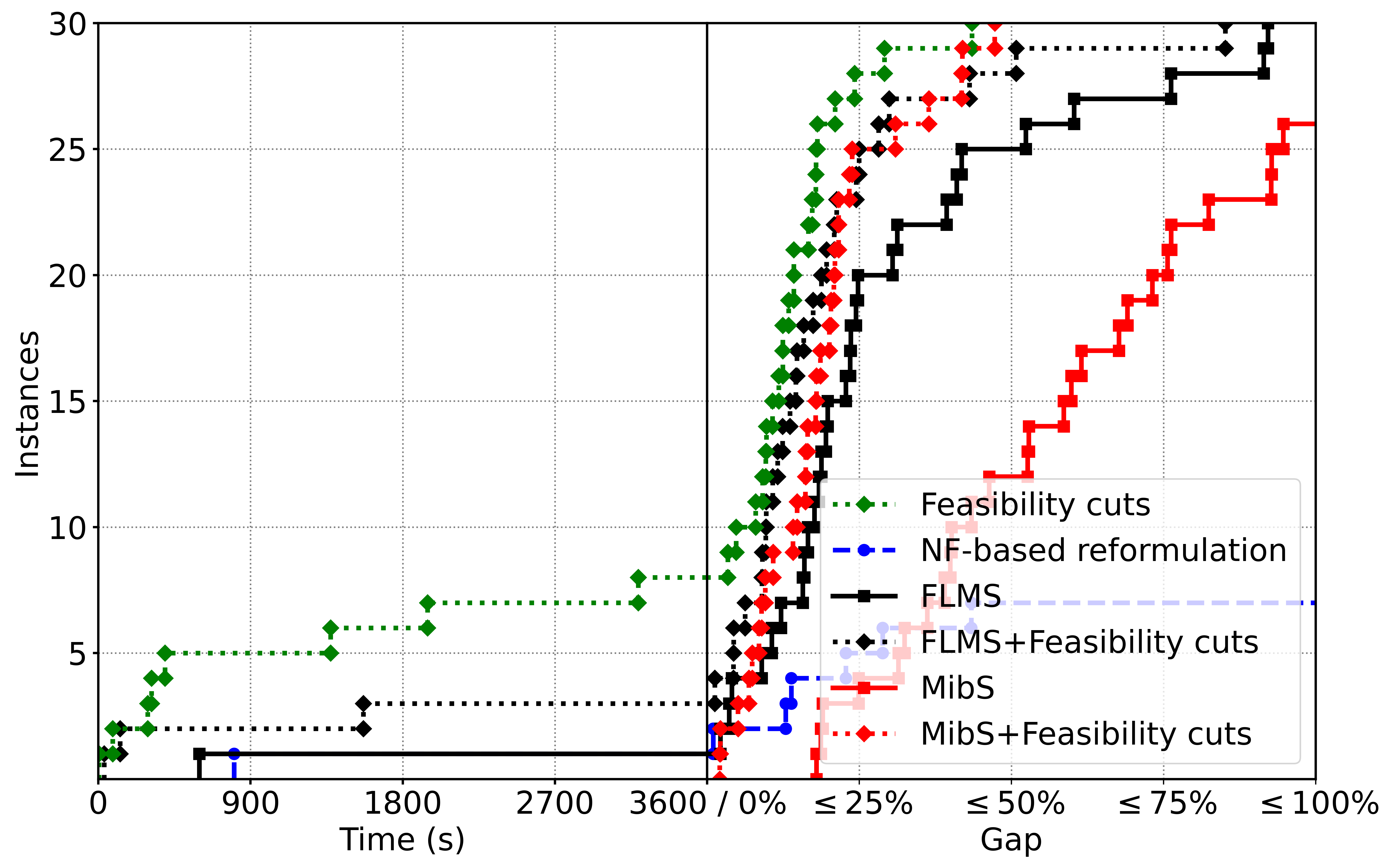}
        \caption{$n = 60,\ k = 5$}
    \end{subfigure}%
    ~
    \begin{subfigure}[t]{0.5\textwidth}
        \centering
        \includegraphics[width=\textwidth]{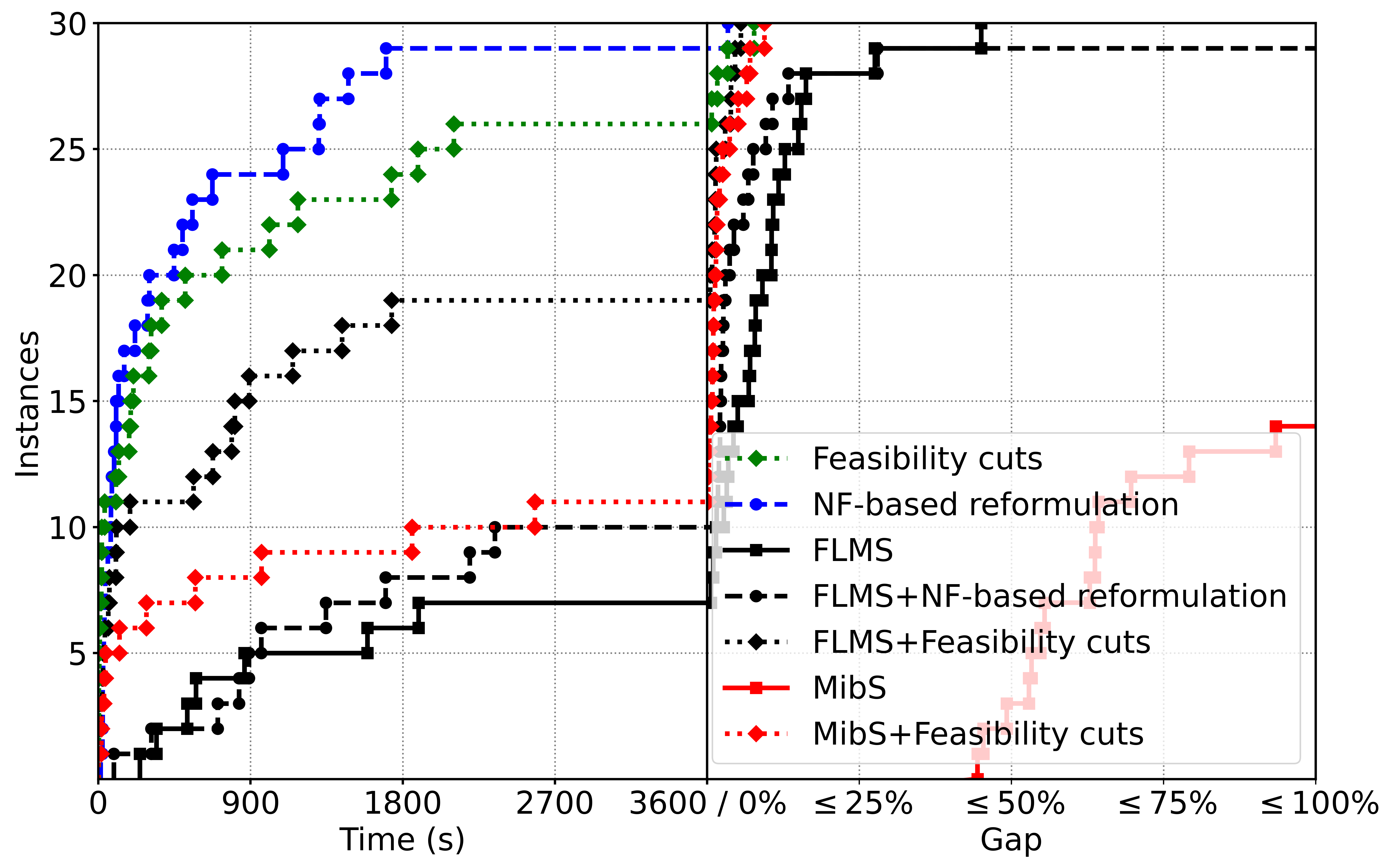}
        \caption{$n = 75,\ k = 1$}
    \end{subfigure}%
    \hfill \\
    \begin{subfigure}[t]{0.5\textwidth}
        \centering
        \includegraphics[width=\textwidth]{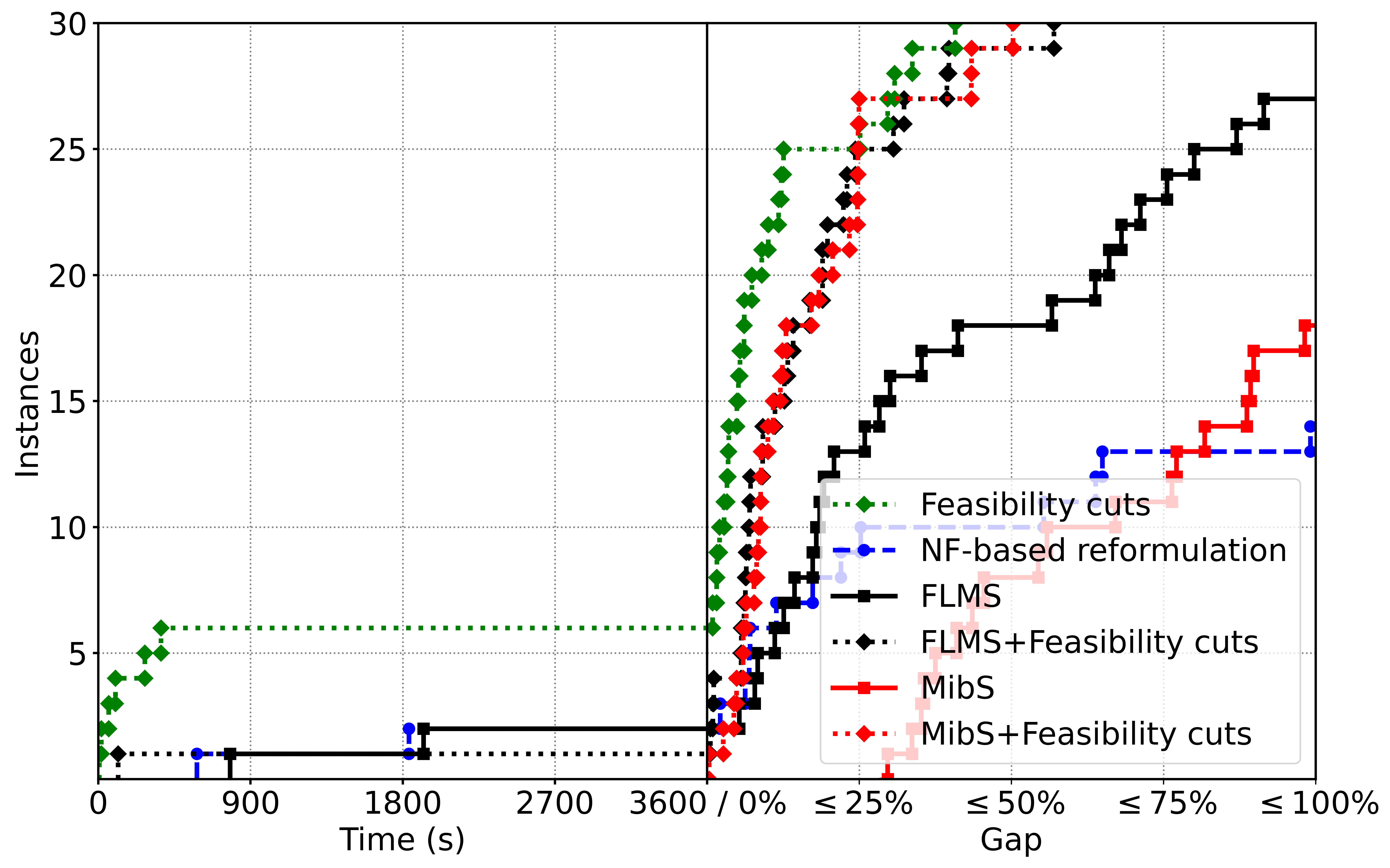}
        \caption{$n = 75,\ k = 3$}
    \end{subfigure}%
    ~
    \begin{subfigure}[t]{0.5\textwidth}
        \centering
        \includegraphics[width=\textwidth]{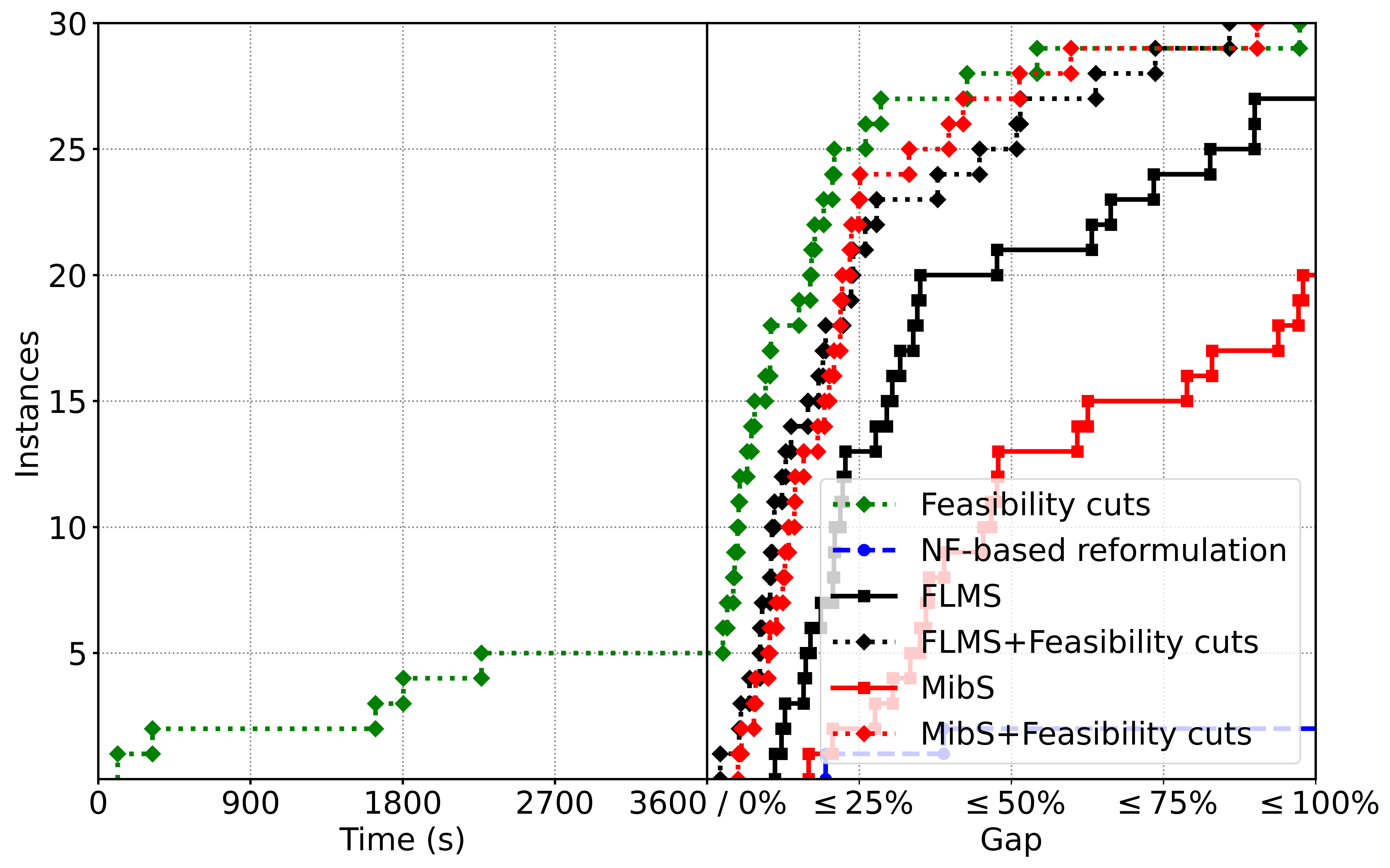}
        \caption{$n = 75,\ k = 5$}
    \end{subfigure}%
    \vspace{.2cm}
    \caption{Results for the block-diagonal follower instances (contd.)}
    \label{fig:results-multirow_part2}
\end{figure*}

\subsection{BOBILib instances}
Finally, we assess the impact of adding feasibility cuts into state-of-the-art bilevel solvers to tackle benchmark instances from BOBILib \citep{thurauf2026bobilib}. We consider the \texttt{cclw} and \texttt{inter-kp} knapsack-interdiction families, the \texttt{imkp} multidimensional-knapsack-interdiction family, and the \texttt{MIPLIB}-derived family. Within these families, we select all instances for which every leader and follower variable is binary and whose BOBILib status is either feasible or open. We exclude the substantially larger or family to prevent a single family from dominating the test set. This selection yields 125 instances.
Because compiling the full networks is intractable, we use approximate value networks capped at 50 nodes per layer. Initial experiments indicate that larger caps produce worse bounds under the same compilation and cut-generation time limit.

\medskip

\noindent \textbf{Comparison.} 
We compare \textsf{FLMS} with its version with added feasibility cuts. In the latter, we compile and generate cuts within 300 seconds, and then execute \textsf{FLMS} for 3,300 seconds.

\medskip

\noindent \textbf{Results.}
Table~\ref{tab:miplib_results} reports the number of solved instances, average solution times over solved instances, and average final gaps over unsolved instances. Feasibility cuts reduce times and gaps, allowing \textsf{FLMS} to close eight additional instances, and up to 13 using a variant of the approximate network compilation procedure in \cite{LBC24}. Appendix \ref{sec:appendix_bobilib_instances_solved} lists the solved instances.

\begin{table}[t]
    \centering
    \caption{Results for BOBILib instances with binary variables.}
    \label{tab:miplib_results}
    \resizebox{\textwidth}{!}{%
    \begin{tabular}{@{}cclrrrlrrr@{}}
        \toprule
        \multirow{2}{*}{Class}               & \multirow{2}{*}{Subclass} &                       & \multicolumn{3}{c}{\textsf{FLMS}} &  & \multicolumn{3}{c}{\textsf{FLMS}+Feasibility cuts} \\ \cmidrule(lr){4-6} \cmidrule(lr){8-10} 
        \multicolumn{2}{c}{} &
           &
          \multicolumn{1}{c}{Solved} &
          \multicolumn{1}{c}{Time (s)} &
          \multicolumn{1}{c}{Gap (\%)} &
           &
          \multicolumn{1}{c}{Solved} &
          \multicolumn{1}{c}{Time (s)} &
          \multicolumn{1}{c}{Gap (\%)} \\ \midrule
        \multirow{2}{*}{Knapsack interdiction}                                   & \texttt{cclw}                 &  & 2/15      & 3,370     & 36.27     &  & 5/15      & 854       & 20.60     \\
                                                                & \texttt{inter-kp}             &  & 0/25     & --        & 32.40     &  & 0/25      & --        & 14.19     \\ \cmidrule{4-10}
        \multirow{2}{*}{Multidimensional knapsack interdiction} & \texttt{imkp}                 &  & 0/5       & --        & 14.94     &  & 2/5       & 1,635     & 8.48      \\
                                                                & \texttt{weish}                &  & 1/47      & 2,829     & 48.43     &  & 1/47      & 66        & 33.33     \\ \cmidrule{4-10}
        MIPLIB                                                 & \multicolumn{1}{l}{} &  & 1/33          &   2,505        &     164.81      &  & 4/33          &   205        &   172.11        \\ \bottomrule
    \end{tabular}%
    }
\end{table}

\section{Conclusions} \label{sec:conclusions}
We develop a cutting-plane framework for binary bilevel programs based on the network encoding of the follower’s value function proposed by \cite{LBC24}. We introduce two exact network flow-based reformulations and separate their associated valid inequalities in the original variable space, strengthening them through tilting. Our computational results reveal a trade-off between embedding the networks in an extended formulation and using them as separation oracles: the former is effective for small networks, whereas the latter performs better as networks become larger or more numerous. On BOBILib instances classified as feasible or open, adding the cuts to \textsf{FLMS} reduces solution times and gaps, closing eight additional instances.
\bibliographystyle{plainnat}
\bibliography{references} 





\clearpage

\appendix

\section{BOBILib instances solved} \label{sec:appendix_bobilib_instances_solved}
Table \ref{tab:bobilib_instances_solved} shows the selected BOBILib instances that are closed by at least one method.

\begin{table}[!ht]
\centering
\caption{BOBILib instances solved. An asterisk (*) indicates that an instance is closed using a slightly different procedure to generate an approximate value network.}
\label{tab:bobilib_instances_solved}
\resizebox{\textwidth}{!}{%
    \begin{tabular}{@{}ccccccc@{}}
        \toprule
        Class & Subclass & Instance &  & \textsf{FLMS} &  & \textsf{FLMS}+Feasibility cuts \\ \midrule
        \multirow{8}{*}{Knapsack interdiction} & \multirow{6}{*}{\texttt{cclw}} & interdiction40-5 &  & $\checkmark$ &  &  \\
         &  & interdiction40-6 &  &  &  & $\checkmark$ \\
         &  & interdiction45-6 &  &  &  & $\checkmark$ \\
         &  & interdiction50-7 &  & $\checkmark$ &  & $\checkmark$ \\
         &  & interdiction55-7 &  &  &  & $\checkmark$ \\
         &  & interdiction55-8 &  &  &  & $\checkmark$ \\ \cmidrule{3-7}
         & \multirow{2}{*}{\texttt{inter-kp}} & K5050W06.KNP &  &  &  & $\checkmark$* \\
         &  & K5050W18.KNP &  &  &  & $\checkmark$* \\ \midrule
        \multirow{7}{*}{Multidimensional knapsack interdiction} & \multirow{4}{*}{\texttt{imkp}} & pet7.dat-0 &  &  &  & $\checkmark$* \\
         &  & pet7.dat-50 &  &  &  & $\checkmark$ \\
         &  & pet7.dat-100 &  &  &  & $\checkmark$ \\
         &  & sento2.dat-100 &  &  &  & $\checkmark$* \\ \cmidrule{3-7}
         & \multirow{3}{*}{\texttt{weish}} & weish14.dat-50 &  & $\checkmark$ &  &  \\
         &  & weish17.dat-50 &  &  &  & $\checkmark$ \\
         &  & weish17.dat-100 &  &  &  & $\checkmark$* \\ \midrule
        \multirow{5}{*}{MIPLIB} &  & lseu-0.500000 &  & $\checkmark$ &  &  \\
         &  & cap6000-0.500000 &  &  &  & $\checkmark$ \\
         &  & p0201-0.100000 &  &  &  & $\checkmark$ \\
         &  & p0282-0.900000 &  &  &  & $\checkmark$ \\
         &  & p0548-0.900000 &  &  &  & $\checkmark$ \\ \bottomrule
    \end{tabular}%
}
\end{table}

\end{document}